\documentclass{amsart}     
\usepackage[all, arc]{xy}
\title{The construction of $E_{\infty}$ ring spaces from bipermutative categories}

\author{J\,P May}
\address{Department of Mathematics\\
The University of Chicago\\
Chicago, Illinois 60637}
\email{may@math.uchicago.edu}
\urladdr{http://www.math.uchicago.edu/~may}

\usepackage{enumerate}
\usepackage{mathrsfs}

\newtheorem{thm}{Theorem}[section]
\newtheorem{cor}[thm]{Corollary}
\newtheorem{prop}[thm]{Proposition}
\newtheorem{lem}[thm]{Lemma}
\newtheorem{conj}[thm]{Conjecture}

\theoremstyle{definition}
\newtheorem{defn}[thm]{Definition}

\newtheorem{con}[thm]{Construction}

\newtheorem{notns}[thm]{Notations}

\theoremstyle{remark}
\newtheorem{rem}[thm]{Remark}
\newtheorem{rems}[thm]{Remarks}

\newtheorem{sch}[thm]{Scholium}

\makeatletter
\let\c@equation\c@thm
\makeatother
\numberwithin{equation}{section}

\DeclareFontFamily{OMS}{rsfs}{\skewchar\font'60}
\DeclareFontShape{OMS}{rsfs}{m}{n}{<-5>rsfs5 <5-7>rsfs7 <7->rsfs10 }{}
\DeclareSymbolFont{rsfs}{OMS}{rsfs}{m}{n}
\DeclareSymbolFontAlphabet{\scr}{rsfs}

\newcommand{\sA}{\scr{A}}
\newcommand{\sB}{\scr{B}}
\newcommand{\sC}{\scr{C}}
\newcommand{\sD}{\scr{D}}
\newcommand{\sE}{\scr{E}}
\newcommand{\sF}{\scr{F}}
\newcommand{\sG}{\scr{G}}

\newcommand{\sI}{\scr{I}}
\newcommand{\sJ}{\scr{J}}
\newcommand{\sK}{\scr{K}}
\newcommand{\sL}{\scr{L}}
\newcommand{\sM}{\scr{M}}
\newcommand{\sN}{\scr{N}}
\newcommand{\sO}{\scr{O}}
\newcommand{\sP}{\scr{P}}
\newcommand{\sQ}{\scr{Q}}

\newcommand{\sT}{\scr{T}}
\newcommand{\sU}{\scr{U}}
\newcommand{\sV}{\scr{V}}
\newcommand{\sW}{\scr{W}}

\newcommand{\bE}{\mathbb{E}}

\newcommand{\bR}{\mathbb{R}}

\newcommand{\bZ}{\mathbb{Z}}

\newcommand{\al}{\alpha}
\newcommand{\be}{\beta}
\newcommand{\ga}{\gamma}
\newcommand{\de}{\delta}
\newcommand{\epz}{\varepsilon}
\newcommand{\ph}{\phi}

\newcommand{\et}{\eta}
\newcommand{\io}{\iota}

\newcommand{\la}{\lambda}
\newcommand{\tha}{\theta}

\newcommand{\rh}{\rho}
\newcommand{\si}{\sigma}
\newcommand{\ta}{\tau}
\newcommand{\ch}{\chi}
\newcommand{\ps}{\psi}
\newcommand{\ze}{\zeta}
\newcommand{\om}{\omega}
\newcommand{\GA}{\Gamma}
\newcommand{\LA}{\Lambda}

\newcommand{\SI}{\Sigma}

\newcommand{\OM}{\Omega}
\newcommand{\XI}{\Xi}
\newcommand{\UP}{\Upsilon}
\newcommand{\PI}{\Pi}

\newcommand{\PH}{\Phi}

\newcommand{\com}{\circ}     
\newcommand{\iso}{\cong}     
\newcommand{\sma}{\wedge}    
\newcommand{\wed}{\vee}      

\newcommand{\rtarr}{\longrightarrow}

\def\quickop#1{\expandafter\newcommand\csname #1\endcsname{\operatorname{#1}}}
\quickop{Hom} \quickop{End} \quickop{Aut} \quickop{Tel} \quickop{Mic}
\quickop{Ext} \quickop{Tor} \quickop{Id} \quickop{Coker} \quickop{Ker}
\quickop{Lim} \quickop{Colim} \quickop{Holim} \quickop{Hocolim}
\quickop{id} \quickop{tel} \quickop{mic} \quickop{coker}
\quickop{colim} \quickop{holim} \quickop{hocolim} \quickop{im}

\newcommand{\DK}{\sK\int\sD}

\begin{document}

\begin{abstract} The construction of $E_{\infty}$ ring spaces and
thus $E_{\infty}$ ring spectra from bipermutative categories gives
the most highly structured way of obtaining the $K$-theory commutative 
ring spectra.  The original construction dates from around 1980 and 
has never been superseded, but the original details are difficult,
obscure, and slightly wrong.  We rework the construction in a much 
more elementary fashion.
\end{abstract}

\maketitle

\tableofcontents

\section*{Introduction}

Bipermutative categories give the most important input into 
multiplicative infinite loop space theory.  The classifying
space of a permutative category is an $E_{\infty}$ space.  We
would like to say that the classifying space of a bipermutative
category is equivalent to an $E_{\infty}$ ring space. That is a 
deeper statement, but it is also true.

My first purported proof of this passage, in \cite{MQR}, was incorrect.
It was based on a nonexistent $E_{\infty}$ operad pair. I wrote the quite
difficult  paper \cite{Mult} to correct this. Although the correction is 
basically correct, there are two rather minor errors of detail in \cite{Mult}
and the paper is quite hard to read.  Fixes for the errors were 
in place in the early 1990's, but were never published.\footnote{The more substantial
fix is purely combinatorial and was given to me by Uwe Hommel in the early 1980's. 
That correction was submitted to JPAA, where \cite{Mult} appeared,
in 1986. The editors declined to publish it since the correction was 
relatively minor and was unreadable in isolation. The introduction of
\cite{EM} exaggerated the errors in \cite{Mult}, which helped spur this
simplified reworking.}
While writing the prequel \cite{Prequel}, I rethought the technical details and 
saw that the easier fix leads to quite elementary ideas that make the harder 
fix unnecessary.  I will give the details here, since they substantially
simplify \cite{Mult}.  In a sense the change is trivial.  The minor errors 
referred to above only concern considerations of basepoints, and I will redo the 
theory in a way that allows the basepoints to take care of themselves, 
following \cite[1.4]{Prequel}. This changes the ground categories of our 
monads to ones made up of unbased spaces, and the change trivializes the
combinatorial descriptions of the relevant monads.

Since the treatment of basepoints is so crucial, we state our conventions
right away.  We consider both based and unbased spaces in \S1 and \S2.
We work solely with unbased spaces in \S\S3--13; we alert the reader 
to a relevant change of notations that is explained at the start of \S3.  
We also fix the convention that when we say that a map is an equivalence, 
we mean that it is a weak homotopy equivalence.

In fact, the mistakes had nothing to do with bipermutative categories.
As I will recall, the work of \cite{MayPer, Mult, Woolf} includes 
two different and entirely correct ways of constructing 
$(\sF\int \sF)$-spaces from bipermutative categories.  This is quite
standard and, by now, quite elementary category theory.  By pullback,
$(\sF\int\sF)$-spaces are $(\hat{\sG}\int \hat{\sC})$-spaces, where
$\hat{\sG}\int \hat{\sC}$ is the category of ring operators associated
to an $E_{\infty}$ operad pair $(\sC,\sG)$.  The
minor errors concerned the construction of $(\sC,\sG)$-spaces, that
is $E_{\infty}$ ring spaces, from $(\hat{\sG}\int \hat{\sC})$-spaces.
With the details here, that construction is now also mainly elementary category theory.

The following diagram will serve as a guide to the revised theory.
It expands the top two lines of the diagram from \cite{Bull} 
that we focused on in the prequel \cite[(0.1)]{Prequel}. 

{\small

\begin{equation}\label{DIAG}
\xymatrix{
 \text{PERM CATS}\ar[d] & & \text{BIPERM CATS} \ar[d]  \\
\sF-\text{CATS} \ar[d]_{B} & & (\sF\int\sF)-\text{CATS}  \ar[d]^{B}\\
 \sF-\text{SPACES} \ar@<1ex>[d] 
& & (\sF\int\sF)-\text{SPACES} \ar@<1ex>[d]  \\
 \hat{\sC}-\text{SPACES}  \ar@<1ex>[dd] \ar@<1ex>[u]& & 
(\hat{\sG}\int\hat{\sC})-\text{SPACES} \ar@<1ex>[d] \ar@<1ex>[u]\\
 & & (\hat{\sC},\hat{\sG})-\text{SPACES} \ar@<1ex>[d] \ar@<1ex>[u]\\
 \sC-\text{SPACES} \ar@{=}[d] \ar@<1ex>[uu]& & (\sC,\sG)-\text{SPACES} 
\ar@{=}[d] \ar@<1ex>[u]\\
 E_{\infty}\ \text{SPACES} & & E_{\infty}\  \text{RING SPACES}   \\}
\end{equation}

}

The intermediate pairs of downwards pointing arrows are accompanied by
upwards pointing arrows, as we will explain, but our focus
is on the downwards arrows, whose bottom targets are the inputs
of the additive and multiplicative black box of the prequel \cite{Prequel}.
We shall work mainly from the bottom of the diagram upwards. 

We review the input of additive infinite loop space theory in 
\S\S\ref{MT1}-\ref{MT3}, which largely follow May and Thomason \cite{MT}.  The central
concept is that of the category of operators $\hat{\sC}$ constructed
from an operad $\sC$. This gives a conceptual intermediary between 
Segal's $\sF$-spaces, or $\GA$-spaces, and $E_{\infty}$ spaces. We
recall this notion in \S\ref{MT1}, and we discuss monads associated
to categories of operators in \S\ref{MT2}.  A key point is to compare
monads on the categories of based and unbased spaces. We give based
and unbased versions of the parallel pair of arrows relating 
$\sF$-spaces and $\hat{\sC}$-spaces in \S\ref{MT1} and \S\ref{MT2}.  
Departing from \cite{MT}, we give an unbased version of the 
parallel pair of arrows relating $\sC$-spaces and $\hat{\sC}$-spaces 
in \S\ref{MT3}. The comparison uses the two-sided monadic bar construction 
that was advertised in \cite[\S8]{Prequel} and used in \cite{MT}, but with 
simplifying changes of ground categories as compared with those used in 
\cite{MT}. 

We then give a parallel review of the input of multiplicative
infinite loop space theory, largely following May \cite{Mult}. Here
we have three pairs of parallel arrows, rather than just two, and 
we need the intermediate category of $(\hat{\sC},\hat{\sG})$-spaces 
that is displayed in (\ref{DIAG}).   This category has two 
equivalent conceptual descriptions,
one suitable for the comparison given by the middle right pair of parallel 
arrows and the other suitable for the comparison given by the bottom right 
pair of parallel arrows.  The equivalence of the two descriptions is 
perhaps the lynchpin of the theory.

We recall the precise definition of an action of one operad on another 
and of one category of operators on another in \S\ref{Mult0}.  We introduce 
categories of ring operators and show that there is a category of ring 
operators $\sJ = \hat{\sG}\int\hat{\sC}$ associated to an operad pair 
$(\sC,\sG)$ in \S\ref{Mult1}.  Actions of $\sJ$ specify the
$(\hat{\sG}\int\hat{\sG})$-spaces of (\ref{DIAG}).  We also elaborate 
the comparison of $\sF$-spaces with $\hat{C}$-spaces given in \S\ref{MT1}
to a comparison of $(\sF\int\sF)$-spaces with $\sJ$-spaces in \S\ref{Mult1}. 
That gives the top right pair of parallel arrows in (\ref{DIAG}).  

We define  $(\hat{\sC},\hat{\sG})$-spaces in \S\ref{MT2too}.  They are intermediate 
between $\sJ$-spaces and $(\sC,\sG)$-spaces, being less 
general than the former and more general than the latter.  To compare
these three notions, we work out the structure of a monad $\bar{J}$ 
whose algebras are the $\sJ$-spaces in \S\ref{barJstruc}.  This 
is where the theory diverges most fundamentally from that of \cite{Mult}. 
We define $\bar{J}$ on a ground category that uses only unbased spaces, 
thus eliminating the need for all of the hard work in \cite{Mult}. This
change also leads to considerable clarification of the conceptual structure 
of the theory.

We use this analysis to construct the middle right pair of parallel arrows 
of (\ref{DIAG}), comparing  $(\hat{\sC},\hat{\sG})$-spaces to $\sJ$-spaces, 
in \S\ref{Bonzai}.  We use it to compare monads on the ground category for
$\sJ$-spaces and on the ground category for $(\hat{\sC},\hat{\sG})$-spaces in
\S\ref{addenda}.  This comparison implies the promised equivalence of our 
two descriptions of the category of $(\hat{\sC},\hat{\sG})$-spaces.  Using
our second description, we construct the bottom right pair of parallel arrows 
of (\ref{DIAG}), comparing  $(\sC,\sG)$-spaces to $(\hat{\sC},\hat{\sG})$-spaces, 
in \S\ref{Mult02}.  This comparison is just a multiplicative 
elaboration of the comparison of $\sC$-spaces and $\hat{\sC}$-spaces in 
\S\ref{MT3}.  

The theory described so far makes considerable use of general categorical 
results about monads and, following Beck \cite{Beck}, about how monads are used 
to encode distributivity phenomena.  These topics are treated in Appendices A and B.

With this theory in place, we recall what permutative and 
bipermutative categories are in \S\ref{Per} and \S\ref{Biper}.  Some examples of 
bipermutative categories will be recalled in the sequel \cite{Sequel}.   There are several variants
of the definition. We shall focus on the original precise definition in order to relate bipermutative categories to $E_{\infty}$ ring spaces most simply, but that is not too important. It is more important that we include topological bipermutative categories, since some of the nicest applications involve the comparison of discrete and topological examples. In line with this, all categories throughout the paper are understood to be topologically enriched and all functors and natural transformations are understood to be continuous. We sometimes repeat this for emphasis, but it is always assumed.  

We explain how to construct $E_{\infty}$ spaces from permutative categories
in \S\ref{Per} and how not to construct $E_{\infty}$ ring spaces from 
bipermutative categories in \S\ref{Biper}.  We recall one of the two correct 
passages from bipermutative categories to $(\sF\int\sF)$-categories in 
\S\ref{Doperad}.  Applying the classifying space functor $B = |N(-)|$,
we obtain $(\sF\int\sF)$-spaces,
from which we can construct $E_{\infty}$ ring spaces.  

There is a more recent foundational theory analogous to that reworked here, 
which is due to Elmendorf and Mandell \cite{EM}.  As recalled in the prequel 
\cite{Prequel}, their work produces (naive) $E_{\infty}$ symmetric spectra, and 
therefore commutative symmetric ring spectra, from (a weakened version of) 
bipermutative categories. Most importantly, they show how to construct algebra 
and module spectra as well as ring spectra from categorical data.\footnote{In
work in progress with Vigleik Angeltveit, we define algebras and modules on
the $E_{\infty}$ space level, which is completely new, and we elaborate the theory 
of this paper to give
a comparison between $E_{\infty}$ rings, modules, and algebras of spaces and
of spectra.} However, their introduction misstates the relationship between their 
work and the 1970's work.  The 1970's applications all depend
on $E_{\infty}$ ring {\em spaces} and not just on commutative ring {\em spectra}.  That is,
they depend on the passage from 
bipermutative categories to $E_{\infty}$ ring spaces, and from there to $E_{\infty}$ ring
spectra. 
Such applications, some of which are summarized in the sequel \cite{Sequel}, are 
not accessible to foundations based on diagram ring 
spectra. We reiterate that a comparison is needed.

\vspace{2mm} 

It is a pleasure to thank Vigleik Angeltveik, John Lind, and
an anonymous referee
for catching errors and suggesting improvements.

\section{Operads, categories of operators, and $\sF$-spaces}\label{MT1}

We review the input data of additive infinite loop space theory, 
since we must build on that to describe the input data for multiplicative
infinite loop space theory.  We first recall the definition of a
category of operators $\sD$ and the construction of a category of 
operators $\hat{\sC}$ from an operad $\sC$.  We then recall the notion 
of a $\sD$-space for a category of operators $\sD$, and finally we
show how to compare categories of $\sD$-spaces as $\sD$ varies.
Aside from the correction of a small but illuminating mistake, 
this material is taken from \cite{MT}, to which we refer the reader 
for further details.

Recall that $\sF$ denotes the category of finite based sets 
$\mathbf{n} =\{0,1,\cdots,n\}$, with $0$ as basepoint,
and based functions.  The category $\sF$ is opposite to 
Segal's category $\GA$ \cite{Seg2}, and $\sF$-spaces are
just $\GA$-spaces by another name.  Let $\PI\subset \sF$ 
be the subcategory whose morphisms are the based functions 
$\ph\colon \mathbf{m} \rtarr \mathbf{n}$ such 
that $|\ph^{-1}(j)| \leq 1$ for $1\leq j\leq n$, 
where $|S|$ denotes the cardinality of a finite set $S$. 
Such maps are composites of injections ($\ph^{-1}(0)=0$) and 
projections ($|\ph^{-1}(j)| = 1$ for $1\leq j\leq n$). The
permutations are the maps that are both injections and
projections.  For an injection 
$\ph\colon \mathbf{m}\rtarr \mathbf{n}$, define
$\SI_{\ph}\subset\SI_n$ to be the subgroup of 
permutations such that\footnote{This is a slight correction of \cite[1.2]{MT}, 
the need for which was observed in \cite[p. 11]{Mult}.}
$\si(\text{Im}\,\ph) = \text{Im}\,\ph$. We shall later
make much use of the subcategory $\UP\subset \PI$ whose morphisms
are the projections.\footnote{$\UP$ is Greek Upsilon and stands for ``unbased'';
we have discarded the injections from $\PI$, keeping only the surjections.
The injections correspond to basepoint insertions in $\PI$-spaces $\{X^n\}$.} 
Note that $\mathbf{0}$ is an initial and terminal object of $\PI$ and of 
$\sF$, giving a map $0$ between any two objects, but it is only a terminal
object of $\UP$. 

\begin{defn}\label{catoper} A category of operators is a 
topological category $\sD$ with objects $\mathbf{n} =\{0,1,\cdots,n\}$, 
$n\geq 0$, such that the inclusion $\PI\rtarr \sF$ factors as the
composite of an inclusion $\PI\subset \sD$ and a surjection 
$\epz\colon \sD\rtarr \sF$, both of which are the identity on objects.  
We require 
the maps
$\sD(\mathbf{q},\mathbf{m}) \rtarr \sD(\mathbf{q},\mathbf{n})$
induced by an injection $\ph\colon \mathbf{m}\rtarr \mathbf{n}$ 
to be $\SI_{\ph}$-cofibrations.  A map $\nu\colon \sD\rtarr \sE$ of 
categories of operators is a continuous functor $\nu$ over $\sF$ and 
under $\PI$. It is an equivalence if each map
$\nu\colon \sD(\mathbf{m},\mathbf{n}) \rtarr \sE(\mathbf{m},\mathbf{n})$
is an equivalence.
\end{defn}

Recall that we understand equivalences to mean weak homotopy equivalences.
More details of the following elementary definition are given in 
\cite[4.1]{MT}; see also Notations \ref{details} below. The cofibration condition 
of the previous definition is automatically satisfied since the maps in 
question are inclusions of components in disjoint unions.  As in \cite{Prequel}, 
we require the $0^{th}$ space of an operad to be a point.

\begin{defn}\label{opercat}  Let $\sC$ be an operad.  Define a category 
$\hat{\sC}$ by letting its objects be the sets $\mathbf{n}$ for $n\geq 0$ 
and letting its space of morphisms $\mathbf{m}\rtarr\mathbf{n}$ be
\[ \hat{\sC}(\mathbf{m},\mathbf{n}) = 
\coprod_{\ph\in\sF(\mathbf{m},\mathbf{n})}\ \prod_{1\leq j\leq n} 
\sC(|\ph^{-1}(j)|); \]
When $n=0$, this is to be interpreted as a point indexed on the unique
map $\mathbf{m}\rtarr \mathbf{0}$ in $\sF$.  Units and composition are 
induced from the unit $\id\in \sC(1)$ and the operad structure maps $\ga$.
If the $\sC(j)$ are all non-empty, $\hat{\sC}$ is a category of operators.
The inclusion of $\PI$ is obtained by using the points $\ast = \sC(0)$ and 
$\id\in \sC(1)$. The surjection to $\sF$ is induced by the projections 
$\sC(j)\rtarr \ast$. 
\end{defn}

\begin{rem}\label{PN}  There is a unique operad $\sP$ such that 
$\sP(0)$ and $\sP(1)$ are each a point and $\sP(j)$ is empty for $j>1$.  
The category $\hat{\sP}$ is $\PI$.  There is also a unique 
operad $\sN$ such that $\sN(j)$ is a point for all $j\geq 0$.  Its 
algebras are the commutative monoids, and $\hat{\sN} = \sF$.  
\end{rem}

\begin{rem}\label{Q} There is a trivial operad $\sQ\subset \sP$ such 
that $\sQ(0)$ is empty (violating our usual assumption), $\sQ(1)$ is a point,
and $\sQ(j)$ is empty for $j>1$.  The category $\hat{\sQ}$ is $\UP$.  Some of
our definitions and constructions will be described in terms of categories of 
operators, although they also apply to more general categories which 
contain $\UP$ but not $\PI$, or which map to $\sF$ but not surjectively.
\end{rem}

\begin{defn}\label{DTSpec} Let $\sD$ be a category of operators.  A 
$\sD$-space $Y$ in $\sT$
is a continuous functor $\sD\rtarr \sT$, written $\mathbf{n}\mapsto Y_n$. 
It is reduced if $Y_0$ is a point.  It is special if the following three
conditions are satisfied.
\begin{enumerate}[(i)]
\item $Y_0$ is aspherical (equivalent to a point).
\item The maps $\de\colon Y_n\rtarr Y_1^n$ induced by 
the $n$ projections $\de_i\colon \mathbf{n}\rtarr \mathbf{1}$,
$\de_i(j) = \de_{i,j}$, in $\UP$ are equivalences.
\item If $\ph\colon \mathbf{m}\rtarr \mathbf{n}$ is an injection,
then $\ph\colon X_m\rtarr X_n$ is a $\SI_{\ph}$-cofibration.
\end{enumerate}
It is very special if, further, the monoid $\pi_0(Y_1)$ is a group.
A map $f\colon Y\rtarr Z$ of $\sD$-spaces is a continuous natural
transformation. It is an equivalence if each $f_n\colon Y_n\rtarr Z_n$
is an equivalence.
\end{defn}

Except for the ``very special'' notion, the definition applies equally
well if we only require $\UP\subset \sD$ and do not require the map to $\sF$ 
to be a surjection. 

\begin{defn}\label{DTsp} Let $\sD[\sT]$ denote the category of 
$\sD$-spaces in $\sT$.
\end{defn}

An $\sF$-space structure on a $\PI$-space $Y$ encodes products.  The 
canonical map $\ph_n\colon \mathbf{n}\rtarr \mathbf{1}$ that sends 
$j$ to $1$ for $1\leq j\leq n$ prescribes a map $Y_0\rtarr Y_1$ when
$n=0$ and a canonical $n$-fold product 
$Y_n\rtarr Y_1$ when $n>0$.  When $Y$ is special, which is the case of interest, 
this product induces a monoid structure on $\pi_0(Y_1)$, and similarly 
with $\sF$ replaced by a general category of operators.  The cofibration
condition (iii) is minor, and a whiskering construction given in 
\cite[App. B]{MT} shows that it results in no loss of generality: 
given a $Y$ for which the condition fails, we can replace it by an
equivalent $Y'$ for which the condition holds.  In fact, the
need for this condition and for the more complicated analogues used 
in \cite{Mult} will disappear from the picture in the next section. 

For a based space $X$, there is a $\PI$-space
$RX$ that sends $\mathbf{n}$ to the cartesian power $X^n$
and in particular sends $\mathbf{0}$ to a point; $RX$ satisfies
the cofibration condition if the basepoint of $X$ is nondegenerate. 
The category $\PI$ encodes the operations that relate the powers of 
a based space.  The specialness conditions on $Y$ state that its 
underlying $\PI$-space behaves homotopically like $RY_1$.
 
There is an evident functor $L'$ from $\PI$-spaces to based spaces
that sends $Y_n$ to $Y_1$.  It was claimed in \cite[1.3]{MT} that $L'$ 
is left adjoint to $R$, but that is false.  There is a unique map 
$\mathbf{0}\rtarr \mathbf{1}$ in $\PI$, and, since $(RX)_0$ is a point,  naturality with respect 
to this map shows that for any map of $\PI$-spaces $Y\rtarr RX$, 
the map $Y_1\rtarr (RX)_1 = X$ must factor through the quotient $Y_1/Y_0$.  
The left adjoint $L$ to $R$ is rather the functor that sends $Y$ to 
$Y_1/Y_0$. In the applications, $Y$ is often reduced, and we could 
restrict attention to reduced $\sD$-spaces at the price of quotienting
out by $Y_0$ whenever necessary.  

Defining $LY = Y_1/Y_0$, we have the adjunction
\begin{equation}\label{LR1}
\PI[\sT](LY, X)\iso \sT(Y,RX).
\end{equation}
This remains true for special $\PI$-spaces and non-degenerately based spaces. 

There is a two-sided categorical bar construction
\[  B(Y,\sD,X) = |B_*(Y,\sD,X)|, \]
where $\sD$ is a small topological category, $X\colon \sD\rtarr \sT$
is a covariant functor, and $Y\colon \sD\rtarr \sT$ is a contravariant
functor \cite[\S12]{Class}.  If $\sO$ is the set of objects of $\sD$,
then the space of $q$-simplices is
\[ Y\times_{\sO}\sD\times_{\sO} \cdots \times_{\sO}\sD\times_{\sO} X \]
or, more explicitly, the disjoint union over tuples of objects $n_i$ in $\sO$ of
\[ Y_{n_q}\times \sD(n_{q-1},n_q)\times \cdots \times \sD(n_0,n_1)\times X_{n_0}.\]
The faces are given by the evaluation maps of $Y$, composition in $\sD$, and
the evaluation maps of $X$.  The degeneracies are given by insertion of 
identity maps. This behaves just like the analogous two-sided bar 
constructions of \cite[\S\S8-9]{Prequel}, and has the same rationale.  
As there, we prefer to ignore model categorical considerations and use
various bar constructions to deal with change of homotopy categories in
this paper.  The following result is \cite[1.8]{MT}.  When specialized
to $\epz\colon \hat{\sC}\rtarr \sF$, it gives the upper left pair of
parallel arrows in (\ref{DIAG}).

\begin{thm}\label{Dchange}  Let $\nu\colon \sD\rtarr \sE$ be an
equivalence of categories of operators.  When restricted to the full
subcategories of special objects, the pullback of action
functor $\nu^*\colon \sE[\sT]\rtarr \sD[\sT]$ induces an 
equivalence of homotopy categories.
\end{thm}
\begin{proof}[Sketch proof] Via $\nu$ and the composition in $\sE$,
each $\sE(-,\mathbf{n})$ is a contravariant functor $\sD\rtarr \sT$; via
the composition of $\sE$, each $\sE(\mathbf{m},-)$ is a covariant 
functor $\sE\rtarr \sT$.  For $Y\in \sD[\sT]$, define 
\[  (\nu_*Y)_n = B(\sE(-,\mathbf{n}),\sD, Y). \]
This gives an ``extension of scalars'' functor $\nu_*\colon \sD\rtarr \sE$.
Notice that $\nu_*Y$ is not reduced even when $Y$ is reduced.
The following diagram displays a natural weak equivalence between 
$Y$ and $\nu^*\nu_* Y$. 
\[ \xymatrix@1{
Y  &  B(\sD,\sD,Y) \ar[rr]^-{B(\nu,\id\id)} 
\ar[l]_-{\epz} & & \nu^* B(\sE,\sD,Y) = \nu^*\nu_*Y. \\} \]
Its left arrow has a natural homotopy inverse $\et$.
Similarly, for $Z\in\sE[\sT]$, the following composite displays a natural
weak equivalence between $\nu_*\nu^* Z$ and $Z$. 
\[ \xymatrix@1{
\nu_*\nu^*Z =B(\sE,\sD,\nu^*Z) \ar[rr]^-{B(\id,\nu,\id)} & &  
B(\sE,\sE,Z) \ar[r]^-{\epz} & Z. \\} \]
The categorically minded reader will notice that these maps should be
viewed as the unit and counit of an adjunction fattened up by the
bar construction.
\end{proof}

While the functor $\nu_*$ takes us out of the subcategory of reduced
objects, we could recover reduced objects by quotienting out $(\nu_*Y)_0$. 
For our present emphasis, all we really care about is the mere existence of 
the functor $\nu^*$, since our goal is to create input for the infinite 
loop space machine that we described in \cite[\S9]{Prequel}.  Thus the distinction is of no great importance.  However, it is thought provoking,
and we show how to eliminate it conceptually in the next section.

\section{Monads associated to categories of operators}\label{MT2}

We are going to change our point of view now, since the change here
in the one operad case will illuminate the more substantial change in the 
two operad case.  We recall the following general and well-known result 
in the form that we gave it in \cite[5.7]{Mult}.  It works in greater 
generality, but the form given there is still our focus here. Since
this by now should be standard category theory known by all algebraic
topologists, we shall not elaborate the details.  We usually write 
$\mu$ and $\et$ generically for the product and unit of monads.

\begin{con}\label{construct}  Let $\sD$ be a topological category 
and let $\XI$ be 
a topologically discrete subcategory with the same objects.  Let $\XI[\sU]$ denote
the category of $\XI$-spaces (functors $\XI\rtarr \sU$) and let 
$\sD[\sU]$ denote the category of $\sD$-spaces (continuous 
functors $\sD\rtarr \sU$). We construct a monad $D$ in 
$\XI[\sU]$ such that $\sD[\sU]$ is isomorphic to the category of 
$D$-algebras in $\XI[\sU]$.  For an object $\mathbf{n}\in \XI$ and a
$\XI$-space $Y$, $(DY)_n$ is the categorical tensor product (or left
Kan extension)
\[  \sD(-,\mathbf{n})\otimes_{\XI} Y. \] 
More explicitly, it is the coequalizer displayed in the diagram
\[\xymatrix@1{
\coprod_{\ph\colon \mathbf{q}\to \mathbf{m}}\sD(\mathbf{m},\mathbf{n})\times Y_q \ar@<.5ex>[r] \ar@<-.5ex>[r] & 
\coprod_{\mathbf{m}} \sD(\mathbf{m},\mathbf{n})\times Y_m 
\ar[r] & \sD(-,\mathbf{n})\otimes_{\XI} Y,\\} \]
where the parallel arrows are given by action maps 
$\XI(\mathbf{q},\mathbf{m})\times Y_q\rtarr Y_m$ and composition maps 
$\sD(\mathbf{m},\mathbf{n})\times \XI(\mathbf{q},\mathbf{m})\rtarr 
\sD(\mathbf{q},\mathbf{n})$. Then
$DY$ is a $\sD$-space (and in particular a $\XI$-space) that extends the 
$\XI$-space $Y$. 
If $Y$ is a $\sD$-space, the inclusion of $\XI$ in $\sD$ induces a map
$DY\rtarr \sD\otimes_{\sD} Y\iso Y$ that gives $Y$ a structure of $D$-algebra,
and conversely. 
\end{con}

The point to be emphasized is that we can use varying subcategories $\XI$
of the same category $\sD$, giving monads on different categories that have
isomorphic categories of algebras.  In the previous section, we considered
a category of operators $\sD$ and focused on $\XI = \PI$.  Then 
Construction \ref{construct} gives the monad $D$ on the category $\PI[\sT]$ that was 
used in \cite{MT}.  In particular, when $\sD = \hat{\sC}$, it gives the monad 
denoted $\hat{C}$ there.  We think of these as reduced monads. Their construction involves the injections in $\PI$, which encode basepoint identifications. 

However, it greatly simplifies the theory here if, when 
constructing a monad associated to a category of operators $\sD$, we 
switch from $\PI$ to its subcategory $\UP$ of projections and so eliminate 
the need for basepoint identifications
corresponding to injections. We emphasize that we do not change $\sD$, 
so that we still insist that it contains $\PI$.  With this switch, 
Construction \ref{construct} specializes to give an ``augmented''
monad $D_+$ on the category $\UP[\sU]$.  In particular, when 
$\sD = \hat{\sC}$, it gives a monad $\hat{C}_+$.  The following definitions
and results show that we are free to use $D_+$ instead of $D$ for our present
purposes; compare Remark \ref{ptofview} below.

\begin{defn} Let $\sD$ be a category of operators.  A $\sD$-space $Y$ in 
$\sU$ is a continuous functor $\sD\rtarr \sU$, written $\mathbf{n}\mapsto Y_n$. 
It is reduced if $Y_0$ is a point.  It is special if the following two
conditions are satisfied.
\begin{enumerate}[(i)]
\item $Y_0$ is aspherical (equivalent to a point).
\item The maps $\de\colon Y_n\rtarr Y_1^n$ induced by 
the $n$ projections $\de_i\colon \mathbf{n}\rtarr \mathbf{1}$,
$\de_i(j) = \de_{i,j}$, are equivalences.
\end{enumerate}
It is very special if, further, the monoid $\pi_0(Y_1)$ is a group.
A map $f\colon Y\rtarr Z$ of $\sD$-spaces is a continuous natural
transformation. It is an equivalence if each $f_n\colon Y_n\rtarr Z_n$
is an equivalence.
\end{defn} 

\begin{defn}\label{DUsp} Let $\sD[\sU]$ 
denote the category of $\sD$-spaces in $\sU$.
\end{defn}

For purposes of comparison, we temporarily adopt the following notations 
for the categories of algebras over the two monads that are obtained from 
$\sD$ by use of Construction \ref{construct}. 

\begin{defn}\label{DUsp2} Let $\sD$ be a category of operators.
\begin{enumerate}[(i)]
\item Let $D_+[\UP,\sU]$ denote the category of algebras
over the monad on $\UP[\sU]$ associated to $\sD$.
\item Let $D[\PI,\sT]$ denote the category of algebras over the monad 
on $\PI[\sT]$ associated to $\sD$.
\end{enumerate}
\end{defn}

In fact, we have two other such categories of algebras over monads in sight.  
One is $D_+[\UP,\sT]$, which is isomorphic to $D[\PI,\sT]$ and to
the category of $\sD$-algebras in $\sT$.  The other is
$D[\PI,\sU]$, which is isomorphic to $D_+[\UP,\sU]$ and to
the category of $\sD$-algebras in $\sU$.   

The situation here is very much like that discussed in 
\cite[\S4]{Prequel}.  If we have an action of $\sD$ on a $\UP$-space
$Y$, then the maps $\mathbf{0}\rtarr \mathbf{n}$ of $\PI\subset \sD$,
together with a choice of basepoint in $Y_0$, give the spaces $Y_n$ 
basepoints.  The injections in $\PI$ also give the unit properties 
of the products on a $\sD$-space $Y$.  Using $\UP$ and $\sU$  
rather than $\PI$ and $\sT$ means that we are not taking the basepoints 
of the $Y_n$ and the analogues of insertion of basepoints induced by the 
injections in $\PI$ as preassigned.  

The following result is analogous to \cite[4.4]{Prequel}. 

\begin{prop}\label{isocatsToo} Let $\sD$ be a category of operators,
such as $\hat{\sC}$ for an operad $\sC$.  Consider the following four categories. 
\begin{enumerate}[(i)]
\item The category $\sD[\sU]$ of $\sD$-spaces in $\sU$.
\vspace{1mm}
\item The category $D_+[\UP,\sU]$ of $D_+$-algebras in $\UP[\sU]$.
\vspace{1mm}
\item The category $\sD[\sT]$ of $\sD$-spaces in $\sT$.
\vspace{1mm}
\item The category $D[\PI,\sT]$ of $D$-algebras in $\PI[\sT]$.
\end{enumerate}
The first two are isomorphic and the last two are isomorphic.  When 
restricted to reduced objects ($Y_0=\ast$), all four are isomorphic.  
In general, the forgetful functor sends $\sD[\sT]$ isomorphically 
onto the subcategory of $\sD[\sU]$ that is obtained by preassigning 
basepoints to $0^{th}$ spaces $Y_0$ and therefore to all spaces $Y_n$.
\end{prop}

We have the analogue of Theorem \ref{Dchange}, with the same proof.

\begin{thm}\label{Dchange2}  Let $\nu\colon \sD\rtarr \sE$ be an
equivalence of categories of operators.  When restricted to the full
subcategories of special objects, the pullback of action
functor $\nu^*\colon \sE[\sU]\rtarr \sD[\sU]$ induces an equivalence
of homotopy categories.
\end{thm}

\section{The comparison between $\sC$-spaces and $\hat{\sC}$-spaces}\label{MT3}

To begin with, let us abbreviate notations from the previous section. 
Let us write $\sV = \UP[\sU]$ for the category of $\UP$-spaces.  This category 
plays a role analogous to $\sU$.  We then write $D[\sV]=D_+[\UP,\sU]$ 
for a category of operators $\sD$.  From here on out, we shall always use augmented monads
rather than reduced ones, and we therefore
drop the $+$ from the notations.  This conflicts with usage in the prequel 
\cite{Prequel} and in all previous work in this area, but hopefully will 
not cause confusion here.  
{\em We will never work in a based context in the rest of this paper.}

Now specialize to $\sD = \hat{\sC}$. Our change of perspective simplifies the passage from $\hat{C}$-spaces to $\sC$-spaces of \cite[\S5]{MT}.  
For an unbased space $X$, define $(RX)_n = X^n$, with the evident 
projections.  For an $\UP$-space $Y$, define $LY = Y_1$.  
Since we have discarded the injection $\mathbf{0}\rtarr \mathbf{1}$ in $\PI$, 
there is no need to worry about the distinction between reduced and 
unreduced $\UP$-spaces, and we have the adjunction
\begin{equation}\label{LR2}
\sV(LY, X)\iso \sU(Y,RX).
\end{equation}
Here the counit of the adjunction is the identity transformation 
$LR \rtarr \Id$, and the unit $\de\colon Y\rtarr RLY$ is given by the maps
$\de\colon Y_n\rtarr Y_1^n$. 
The first of the following observations is repeated from \cite[5.2--5.4]{MT}, and 
the second follows by inspection.  The reader may wish to
compare the second with the analogous but more complicated result 
\cite[5.5]{MT}, which used $\PI$ and $\sT$ instead of $\UP$ and $\sU$. 

\begin{notns}\label{effect}  A morphism $\ps$ in $\sF$ is effective if 
$\ps^{-1}(0) = 0$; thus the effective morphisms in $\PI$ are the injections,
including the injections $0\colon \mathbf{0}\rtarr \mathbf{n}$ for $n\geq 0$.  
An effective morphism $\ps$ is ordered if $\ps(i) < \ps(j)$ implies $i < j$. Let 
$\sE\subset \sF$ denote the subcategory of objects $\{\mathbf{n}\}$ and ordered 
effective morphisms $\ps$.  
\end{notns}

\begin{lem}\label{describe}  Any morphism $\ph$ in $\sF$ factors as a composite $\ps\com \pi$, where $\pi$ is a projection and $\ps$ is effective, uniquely up to a permutation of the source of $\psi$. If 
$\ps\colon \mathbf{m}\rtarr \mathbf{n}$ is effective, there 
is a permutation $\ta\in \SI_m$ such that $\ps\com\ta$ is ordered. If $\ps$ is ordered, then $\ps\com\ta$ is also ordered if and only if 
$\ta\in \SI(\ps)\subset \SI_m$, where 
$\SI(\ps) = \SI_{r_1}\times\cdots \times \SI_{r_n}$, $r_j = |\ps^{-1}(j)|$.
\end{lem}

\begin{lem}\label{5.5} For an $\UP$-space $Y$, $(\hat{C}Y)_0 = Y_0$ and, for
$n\geq 1$,
\[ (\hat{C}Y)_n = \coprod_{\ps\in\sE(\mathbf{m},\mathbf{n})}\ 
(\prod_{1\leq j\leq n}\sC(|\ps^{-1}(j)|) )\times_{\SI(\ps)}Y_m. \]
\end{lem}

The following analogues of \cite[5.6 -- 5.8]{MT} are now easy.  Since we have
performed no gluings along injections, at the price of retaining factors 
$\sC(0)$ in the description of $\hat{C}Y$, no cofibration conditions are required.

\begin{lem}\label{5.6}  
Assume that $\sC$ is $\SI$-free, in the sense that each $\sC(j)$ is
$\SI_j$-free.  If $f\colon Y\rtarr Y'$ is an equivalence of $\UP$-spaces, then
so is $\hat{C}f$.  
\end{lem}

Recall the monad $C^{\sU}_+$ on $\sU$ from \cite[4.1]{Prequel}. In line with
the conventions at the beginning of this section, we abbreviate notation to $C$ in
this paper, so that
\begin{equation}\label{CX} 
CX = \coprod_{m\geq 0} \sC(m)\times_{\SI_m} X^m.
\end{equation}
Here and below, we must remember that the empty product of spaces is
a point.  For $m\geq 0$, $\ph_m$ is the unique effective morphism 
$\mathbf{m}\rtarr \mathbf{1}$ (which is automatically ordered), and 
the following result is clear.

\begin{lem}\label{5.7}  
Let $X\in\sU$.  Then $L\hat{C} RX \equiv (\hat{C}RX)_1 = CX$,
and the natural map $\de\colon \hat{C}RX \rtarr RL\hat{C}RX =RCX$ 
is an isomorphism.
\end{lem}

\begin{lem}\label{5.8} Assume that $\sC$ is $\SI$-free.  If $Y$ is a special 
$\UP$-space, then so is $\hat{C}Y$, hence $\hat{C}$ restricts to a monad 
on the category of special $\UP$-spaces.
\end{lem}
\begin{proof}  Applying Lemma \ref{5.6} to the horizontal arrows in the 
commutative diagram
\[ \xymatrix{
\hat{C}Y \ar[r]^-{\hat{C}\de} \ar[d]_{\de} & \hat{C}R L Y
\ar[d]_{\iso}^{\de} \\
RL\hat{C}Y \ar[r]_-{RL\hat{C}\de} & RL\hat{C}RL Y,\\} \]
we see that its left vertical arrow is an equivalence.
\end{proof}

We can now compare $\hat{\sC}$-spaces in $\sV$ to $\sC$-spaces in $\sU$ in 
the same way that we compared the analogous categories of based spaces in 
\cite[p.\, 219]{MT}.  We use the two-sided monadic bar construction of 
\cite{Geo}, the properties of which are recalled in \cite[\S8]{Prequel}.
We recall relevant generalities relating monads to adjunctions in 
Appendix A.  We use properties of geometric realization proven in 
\cite{Geo} and the following unbased analogue of \cite[12.2]{Geo}, which has
essentially the same proof.

\begin{lem}\label{geohat} For simplicial objects $Y$ in the category $\sV$, there is a natural 
isomorphism $\nu\colon |\hat{C}Y|\rtarr \hat{C}|Y|$ such that the
following diagrams commute.
\[\xymatrix{
|Y| \ar[r]^-{|\et|} \ar[dr]_{\et} & |\hat{C}Y| \ar[d]^{\nu}\\
 & \hat{C}|Y| \\ }
\ \ \  \text{and}\ \  \
\xymatrix{
|\hat{C}\hat{C}Y| \ar[r]^-{|\mu|} \ar[d]_{\hat{C}\nu\com\nu} 
& |\hat{C}Y | \ar[d]^{\nu} \\
\hat{C}\hat{C}|Y| \ar[r]_{\mu} &  \hat{C}Y \\ }  \]
If $(Y,\xi)$ is a simplicial $\hat{C}$-algebra, then $(|Y|,|\xi|\com\nu^{-1})$
is a $\hat{C}$-algebra.
\end{lem}

\begin{thm}\label{ChatC}  If $\sC$ is $\SI$-free, then the functor $R$ induces 
an equivalence from the homotopy category of $\sC$-spaces to the homotopy category of special $\hat{\sC}$-spaces.
\end{thm}
\begin{proof} Lemma  \ref{5.7} puts us into one of the two 
contexts discussed in general categorical terms in Proposition \ref{omniold}.  Let $X$ be
a $\sC$-space and $Y$ be a $\hat{\sC}$-space.  By (iii) and (iv) of 
Proposition \ref{omniold}, $R$ embeds the category of $\sC$-spaces as the full 
subcategory of the category of $\hat{\sC}$-spaces consisting of those
$\hat{\sC}$-spaces with underlying $\UP$-space of
the form $RX$.  By (i) and (ii) of Proposition \ref{omniold}, $CL$ is a 
$\hat{C}$-functor and we can define a functor 
$\LA\colon \hat{\sC}[\sU]\rtarr \sC[\sU]$ by
\[ \LA Y = B(CL, \hat{C},Y). \]
By Corollaries \ref{coromni1} and \ref{coromni2}, together with general 
properties of the geometric realization of simplicial spaces proven in 
\cite{Geo}, we have a diagram
\[ \xymatrix@1{
Y & B(\hat{C},\hat{C},Y)\ar[r]^-{\de} \ar[l]_-{\epz} 
& B(RCL,\hat{C},Y)\iso R\LA Y\\} \]
of $\hat{\sC}$-spaces in which the map $\epz$ is a homotopy equivalence with natural homotopy inverse $\et$ and the map $\de = B(\de,\id,\id)$ is an equivalence when $Y$ is special.  Thus the diagram displays a natural weak equivalence between 
$Y$ and $R\LA Y$.  When $Y=RX$, the displayed diagram is obtained by applying 
$R$ to the analogous diagram
\[ \xymatrix@1{
X & B(C,C,X)\ar[r]^-{\iso} \ar[l]_-{\epz} & B(CL,\hat{C},RX) = \LA RX\\} \]
of $C$-algebras, in which $\epz$ is a homotopy equivalence with natural
inverse $\et$. 
\end{proof}

\begin{rem}\label{ptofview} In \cite{MT}, the focus was on the 
generalization of the infinite loop space machine of \cite{Geo} from 
$\sC$-spaces in $\sT$ to $\hat{\sC}$-spaces in $\sT$. For that purpose, 
it was essential 
to use the approximation theorem and therefore essential to use the monads 
in $\sT$ and $\PI[\sT]$ that
are constructed using basepoint type identifications.  It is that 
theory that forced the use of the cofibration condition Definition  \ref{DTSpec}(iii).
However, we are here only concerned with the conversion of 
$\hat{\sC}$-spaces to $\sC$-spaces, and for that purpose we are
free to work with the simpler monads on $\sU$ and $\sV=\UP[\sU]$ whose
algebras are the $\sC$-spaces and $\hat{\sC}$-spaces in $\sU$.
From the point of view of infinite loop space machines, we prefer
to convert input data to $\sC$-spaces and then apply the original
machine of \cite{Geo} rather than to generalize 
the machine to $\hat{\sC}$-spaces.
\end{rem}

\section{Pairs of operads and pairs of categories of operators}\label{Mult0}

With this understanding of the additive theory, we now turn to the
multiplicative theory.  We first recall some basic definitions from 
\cite[\S1]{Mult} since they are 
essential to understanding the details. However, the reader should not
let the notation obscure the essential simplicity of the ideas. We are
just parametrizing the structure of a ring space, or more accurately
rig space since their are no negatives, and then generalizing from 
operations on products $X^n$ to operations on $Y_n$, where, when $Y$ is 
special, $Y_n$ looks homotopically like $Y_1^n$.

The category $\sF$ is symmetric monoidal (indeed, bipermutative) under the 
wedge and product.  On 
objects, the operations are sum and product interpreted by 
ordering elements in blocks and lexicographically.  That is, the set
$\mathbf{m}\wed \mathbf{n}$ is identified with $\mathbf{m+n}$ 
by identifying $i$ with $i$ for $1\leq i\leq m$ and $j$ with
$j+m$ for $1\leq j\leq n$, and the set $\mathbf{m}\sma \mathbf{n}$
is identified with $\mathbf{mn}$ by identifying $(i,j)$, $1\leq i\leq m$
and $1\leq j\leq n$ with $ij$, with the ordering $ij < i'j'$ if $i < i'$ 
or $i=i'$ and $j< j'$. The wedge and smash product of morphisms are forced 
by these identifications. We fix notations for standard permutations.

\begin{notns}\label{perms1} Fix non-negative integers $k$, $j_r$ for $1\leq r\leq k$, 
and $i_{r,q}$ for $1\leq r\leq k$ and $1\leq q\leq j_r$.
\begin{enumerate}[(i)]
\item Let $\si\in \SI_k$.  Define $\si\langle j_1,\dots, j_k \rangle$
to be that permutation of $j_1\cdots j_k$ elements which corresponds under lexicographic identification to the permutation of smash products
\[ \si\colon \mathbf{j_1}\sma \cdots \sma \mathbf{j_k}
\rtarr \mathbf{j_{\si^{-1}(1)}}\sma \cdots \sma \mathbf{j_{\si^{-1}(k)}}. \]
\item Let $\ta_r\in \SI_{j_r}$, $1\leq r\leq k$.  Define 
$\ta_1\otimes \cdots \otimes \ta_k$ 
to be that permutation of $j_1\cdots j_k$ elements which corresponds under lexicographic identification to the smash product of permutations
\[ \ta_1\sma\cdots \sma \ta_k\colon \mathbf{j_1}\sma \cdots \sma \mathbf{j_k}
\rtarr \mathbf{j_1}\sma \cdots \sma \mathbf{j_k}. \]
\item Let $Q$ run over the set of sequences $(q_1,\cdots, q_k)$ such that 
$1\leq q_r \leq j_r$, ordered lexicographically.  Define 
$\nu = \nu(\{k,j_r,i_{r,q}\})$ to be that permutation of 
\[ \Sigma_{Q}\, (\times_{1\leq r\leq k}\, i_{r,q_r}) 
= \times_{1\leq r\leq k}\, (\Sigma_{1\leq q\leq j_r}\, i_{r,q}) \]
elements which corresponds under block sum and lexicographic 
identifications on the left and right to the natural distributivity
isomorphism
\[ \bigvee_{Q}\, (\bigwedge_{1\leq r\leq k}\, \mathbf{i_{r,q_r}}) \iso 
\bigwedge_{1\leq r\leq k}\, (\bigvee_{1\leq q\leq j_r}\, \mathbf{i_{r,q}}). \]
\end{enumerate}
\end{notns}

\begin{defn}\label{CLact} Let $\sC$ and $\sG$ be operads with
$\sC(0)=\{0\}$ and $\sG(0) =\{1\}$. Write $\ga$ for the structure
maps of both operads and $\id$ for the unit elements in both $\sC(1)$ 
and $\sG(1)$.   An action of $\sG$ on $\sC$ consists of maps 
\[\la \colon \sG(k)\times \sC(j_1)\times \cdots  \times \sC(j_k)
\rtarr \sC(j_1\cdots j_k) \]
for $k\geq 0$ and $j_r\geq 0$ which satisfy the following distributivity,
unity, equivariance, and nullity properties.  Let 
\[ g\in \sG(k)\ \ \  \text{and} \ \ \  g_r\in \sG(j_r) \ \ \text{for}\ \
1\leq r\leq k \]
\[ c\in \sC(j)\ \ \  \text{and} \ \ \  c_r\in \sC(j_r) \ \ \text{for}\ \,
1\leq r\leq k \]
\[ c_{r,q} \in \sC(i_{r,q}) \ \ \text{for}\ \ 1\leq r\leq k \ \ \text{and}
\ \ 1\leq q\leq j_r. \]
Further, let
\[ c_{J_r} = (c_{r,1},\cdots, c_{r,j_r}) 
\in \sC(i_{r,1})\times \cdots \times \sC(i_{r,j_r})\]
and
\[ c_Q = (c_{1,q_1},\cdots, c_{k,q_k})\in \sC(i_{1,q_1})\times\cdots \times 
\sC(i_{k,q_k}). \]

\begin{enumerate}[(i)]
\item $\la(\ga(g;g_1,\cdots,g_k); c_{J_1},\cdots, c_{J_k})
= \la(g;\la(g_1; c_{J_1}), \cdots, \la(g_k; c_{J_k}))$.
\vspace{.5mm}
\item $\ga(\la(g;c_1,\cdots,c_k);\times_{Q}\la(g;c_Q))\ \nu
= \la(g;\ga(c_1;c_{J_1}),\cdots, \ga(c_k;c_{J_k}))$.
\vspace{.5mm}
\item $\la(\id;c) = c$.
\vspace{.5mm}
\item $\la(g;\id^k) = \id$.
\vspace{.5mm}
\item $\la(g\si; c_1,\cdots,c_k) = 
\la(g;c_{\si^{-1}(1)}, \cdots, c_{\si^{-1}(k)})\
\si\langle j_1,\dots, j_k \rangle$.
\vspace{.5mm}
\item $\la(g; c_1\ta_1,\cdots,c_k\ta_k) =
\la(g;c_1,\cdots,c_k)\ \ta_1\otimes \cdots \otimes \ta_k$.
\vspace{.5mm}
\item $\la(1) = \id\in \sC(1)$ when $k=0$.
\vspace{.5mm}
\item $\la(g;c_1,\cdots,c_k) = 0$ when any $j_r = 0$. 
\end{enumerate}
\end{defn}

Here (i), (iii), (v), and (vii) relate the $\la$ to the internal structure 
of $\sG$, while (ii), (iv), (vi), and (viii) relate the $\la$ to the internal
structure of $\sC$. 

We have an analogous notion of an action of a category of operators
$\sK$ on a category of operators $\sD$.  Again, we fix notations
for some standard permutations.  

\begin{notns}\label{perms2} Let $\ph\colon \mathbf{m} \rtarr \mathbf{n}$ and 
$\ps\colon \mathbf{n} \rtarr \mathbf{p}$ be morphisms in $\sF$. For
nonnegative integers $r_i$, $1\leq i\leq m$, define 
$s_k = \times_{\ps \ph(i) = k} r_i$.  Define $\si_k(\ps,\ph)$ to be that
permutation of $s_k$ letters which corresponds under lexicographic
ordering to the bijection
\[ \bigwedge_{\ps\ph(i) = k}\, \mathbf{r_i} \rtarr 
\bigwedge_{\ps(j) = k} \, \bigwedge_{\ph(i) = j} \mathbf{r_i} \]
that permutes the factors $\mathbf{r_i}$ from their order on the left
($i$ increasing) to their order on the right ($j$ increasing and, for
fixed $j$, $i$ increasing).  Here $s_k = 1$ and 
$\si_k(\ps,\ph)\colon \mathbf{1}\rtarr \mathbf{1}$ is the identity
if there are no $i$ such that $\ps\ph(i) = k$.  Define $\si(\ps,\ph)$
to be the isomorphism in $\PI^p$ with coordinates the $\si_k(\ps,\ph)$. 
For morphisms $f\colon \mathbf{m} \rtarr \mathbf{n}$ and 
$g\colon \mathbf{n} \rtarr \mathbf{p}$ in a category of operators
$\sD$, write $\si_k(g,f) = \si_k(\epz(g),\epz(f))$, and write $\si(g,f)$
for their product in $\sD^p$. 
\end{notns}

Let $\sD^0$ be the trivial category, which has one object $\ast$ and its identity morphism. 

\begin{defn}\label{CLhatact} Let $\sD$ and $\sK$ be categories of
operators.  An action $\la$ of $\sK$ on $\sD$ consists of functors
$\la(f)\colon \sD^m\rtarr \sD^n$ for $f\in\sK(\mathbf{m},\mathbf{n})$
which satisfy the following properties. Let $\epz(f) = \ph\colon \mathbf{m}\rtarr \mathbf{n}$.
\begin{enumerate}[(i)]
\item On objects, $\la(f)$ is specified by
\[\la(f)(\mathbf{r_1},\cdots,\mathbf{r_m}) 
= (\mathbf{s_1},\cdots,\mathbf{s_n}), \ \ \ \text{where} \ \ \
\mathbf{s_j} = \sma_{\ph(i) = j}\, \mathbf{r_i}. \]
\item On morphisms $(\ch_1,\cdots,\ch_m)$ of $\PI^m\subset \sD^m$,
$\la(f)$ is specified by
\[ \la(f)(\ch_1,\cdots,\ch_m) = (\om_1,\cdots,\om_n), \ \ \ 
\text{where}\ \ \ \om_j = \sma_{\ph(i)=j}\, \ch_i.  \]
\item On general morphisms $(d_1,\cdots, d_m)$ of $\sD^m$, $\la(f)$ 
satisfies
\[ \epz(\la(f)(d_1,\cdots,d_m)) = (\om_1,\cdots,\om_n), \ \ \ 
\text{where}\ \ \ \om_j = \sma_{\ph(i)=j}\, \epz(d_i).  \]
\item For morphisms $\ph\colon \mathbf{m}\rtarr \mathbf{n}$ of 
$\PI\subset \sK$, $\la(\ph)$ is specified by
 \[ \la(\ph)(d_1,\cdots,d_m) = (d_{\ph^{-1}(1)},\cdots,d_{\ph^{-1}(n)}) \]
\item For morphisms $f\colon \mathbf{m}\rtarr \mathbf{n}$ and
$g\colon \mathbf{n}\rtarr \mathbf{p}$ in $\sK$, the 
isomorphisms $\si(g,f)$ in $\PI^p\subset \sC^p$ specify a natural
isomorphism $\la(g\com f)\rtarr \la(g)\com \la(f)$.
\end{enumerate}
If $\ph^{-1}(j)$ is empty, then the $j^{th}$ coordinate of $\la(f)$ is 
$\mathbf{1}$ in (i) and the $j^{th}$ coordinate is $\text{id}\in\sC(1)$
in (ii)--(iv). (Compare Definition \ref{CLact}(vii)).
\end{defn} 

In what should by now be standard bicategorical language, the $\la(\mathbf{n})$, 
$\la(f)$, and $\si(g,f)$ specify a pseudofunctor $\la\colon \sK\rtarr \sC\!at$.
We do not assume familiarity with this, but it shows that the definition is
sensible formally. The
definition itself specifies an action of $\PI$ on any category  of 
operators $\sD$ and an action of any category of operators $\sK$
on both $\PI$ and $\sF$.  However, our interest is in $(\hat{\sC},\hat{\sG})$,
where $(\sC,\sG)$ is an operad pair with $\sG$ acting on $\sC$. To connect up
definitions, we first use Notations \ref{perms2} to recall how composition is defined
in the category of operators $\hat{\sC}$ associated to an operad $\sC$.

\begin{notns}\label{details} For an operad $\sC$, write 
$(\ph;c_1,\cdots,c_k)$, or $(\ph;c)$ for short, for morphisms in 
$\hat{\sC}(\mathbf{m},\mathbf{n})$. Here 
$\ph\colon \mathbf{m} \rtarr \mathbf{n}$ is a morphism in $\sF$
and $c_j\in \sC(|\ph^{-1}(j)|)$, with $c_j = 0\in \sC(0)$ if
$\ph^{-1}(j)$ is empty.  For $(\ps;d)\in \hat{\sC}(\mathbf{n},\mathbf{p})$, composition in $\hat{\sC}$ is specified by
\[ (\ps;d)\com (\ph;c) = (\ps\com\ph;\times_{1\leq k\leq p}\
\ga(d_k;\times_{\ps(j) = k} c_j)\, \si_k(\ps,\ph)). \]
\end{notns}

\begin{notns}\label{moredetails}  Recall that we have canonical morphisms 
$\ph_n\colon \mathbf{n} \rtarr \mathbf{1}$ in $\sF$ that send $j$ to $1$ 
for $1\leq j\leq n$.  Together with the morphisms of $\PI$, they generate
$\sF$ under the wedge sum.  Notice that 
$\sma_{1\leq r\leq k} \ph_{j_r} = \ph_{j_1\cdots j_r}$.  We define an
embedding $\io$ of the operad $\sC$ in the category of operators $\hat{\sC}$
by mapping $c\in \sC(n)$ to the morphism
$(\ph_n;c)\colon \mathbf{n}\rtarr \mathbf{1}$.  Using wedges in $\sF$ and
cartesian products of spaces $\sC(j)$, we define maps
\[ \hat{\sC}(\mathbf{j_1},1)\times \cdots \times \hat{\sC}(\mathbf{j_k},1)
\rtarr \hat{\sC}(\mathbf{j_1 +\cdots + j_k},\mathbf{k}). \]
The operadic structure maps $\ga$ are recovered from these maps and composition
\[\hat{\sC}(\mathbf{k},\mathbf{1})\times 
\hat{\sC}(\mathbf{j_1 +\cdots + j_k},\mathbf{k})
\rtarr \hat{\sC}(\mathbf{j_1 +\cdots + j_k},\mathbf{1}). \]
\end{notns}

The following result is \cite[1.9]{Mult}, and more details may be found there.

\begin{prop} An action $\la$ of an operad $\sG$ on an operad $\sC$
determines and is determined by an action of $\hat{\sG}$ on $\hat{\sC}$.
\end{prop}
\begin{proof}[Sketch proof]  We have the embeddings $\io$ of $\sC$ in
$\hat{\sC}$ and $\sG$ in $\hat{\sG}$.  An action $\la$ of $\sG$ on $\sC$ 
is related to the corresponding action $\la$ of $\hat{\sG}$ on $\hat{\sC}$ by
\begin{equation}\label{lambda2}
\io\la(g;c_1,\cdot,c_k) = \la(\io g;\io c_1,\cdots, \io c_k).
\end{equation}
Given $\la$ on the categories, this clearly determines $\la$
on the operads.  Conversely, given the combinatorics of how $\hat{\sG}$ 
and $\hat{\sC}$ are constructed from $\sG$ and $\sC$, there is a unique way 
to extend (\ref{lambda2}) from the operads to the categories.  Looking at 
Definition \ref{CLhatact}, we see that if $f\colon \mathbf{m}\rtarr \mathbf{n}$ is 
a morphism of $\hat{\sG}$ with $\epz(f) = \ph$ and $c_i$ is a morphism of 
$\hat{\sC}$,
$1\leq i\leq m$, then the $j^{th}$ coordinate of $\la(f;c_1,\cdots,c_m)$
depends only on those $c_i$ with $\ph(i) = j$, and $f$ has coordinates
$f_j\in \sG(|\ph^{-1}(j)|)$ that allow use of the operadic $\la$ to
specify the categorical $\la$. Details are in \cite[1.9]{Mult}.  
Formulas (i), (iii), and (v) of Definition \ref{CLact} correspond to the requirement 
that the $\la(f)$ be functors. Formulas (ii), (iv), and (vi) correspond to 
the naturality requirement of Definition \ref{CLhatact}(v).  Formulas (vii) and (viii)
are needed for compatible treatment of $1\in\sG(0)$ and $0\in\sC(0)$. 
\end{proof}

\section{Categories of ring operators and their actions}\label{Mult1}

We can coalesce a pair of operator categories $(\sD,\sK)$ into a
single wreath product category $\DK$.  The construction actually 
applies to any pseudofunctor $\la$ from any category $\sG$ to
$\sC\!at$, but we prefer to specialize in order to fix notations.

\begin{defn}\label{wreath} 
Let $\la$ be an action of $\sK$ on $\sD$, where $\sK$ and $\sD$
are categories of operators.   The objects of $\DK$ are the $n$-tuples 
of finite based sets (objects of $\sF$) for $n\geq 0$.  We write objects as 
$(n;S)$, where $S = (\mathbf{s_1},\cdots,\mathbf{s_n})$. There is a single object, denoted $(0;\ast)$, when $n=0$; we think of $\ast$ as the
empty sequence.  The space of morphisms 
$(m;R)\rtarr (n;S)$ in $\DK$ is
\[ \coprod_{\ph \in \sF (\mathbf{m},\mathbf{n})} \epz^{-1}(\ph)\times \prod_{1 \leq j\leq n}\sD(\bigwedge_{\ph (i)=j}\mathbf{r_i}, \mathbf{s_j}),  \ \ \ \epz\colon \sK\rtarr \sF, \]
where the empty smash product is $\mathbf{1}$. Typical morphisms are written 
$(f;d)$, where $f\in\sK(\mathbf{m},\mathbf{n})$
and $d=(d_1,\cdots,d_n)$. If $\epz(f) = \ph$, then 
$d_j\in \sD(\sma_{\ph(i)=j}\mathbf{r_i},\mathbf{s_j})$.  
For a morphism $(g;e)\colon (n;S)\rtarr (p;T)$,
composition is specified by
\[ (g;e)\com (f;d) = (g\com f;e\com \la(g)(d)\com \si(g,f)). \]   
More explicitly, with $\epz(g) = \ps$, the $k^{th}$ coordinate of 
$e\com \la(g)(d)\com \si(g,f)$
is the composite
\[ \xymatrix@1{
\bigwedge_{\ps\ph(i)=k}\, \mathbf{r_i} \ar[r]^-{\si_k(\ps,\ph)}
& \bigwedge_{\ps(j)=k}\, \bigwedge_{\ph(i) = j}\, \mathbf{r_i} 
\ar[rrr]^-{\la_k(g)(\times_{\ps(j) = k}d_j)}
& & & \bigwedge_{\ps(j)=k} \mathbf{s_j} \ar[r]^-{e_k} & \mathbf{t_k}. \\} \]
The object $(0;\ast)$ is terminal, with unique morphism $(m;R)\rtarr (0;\ast)$
denoted $(0;\ast)$;  the morphisms $(0;\ast)\rtarr 
(n;S)$ are of the form 
$(0;d)=(\text{id};d)\com (0;\text{id}^n)$, where 
$0\colon \mathbf{0}\rtarr \mathbf{n}$, $\id\colon \mathbf{n}\rtarr \mathbf{n}$
in $\sF$ on the left, and $\id^n\in\sC(1)^n$ on the right.
\end{defn}

We write the morphisms of $\PI\int\PI$ in the form $(\ph;\ch)$, where 
\[ \ch = (\ch_1,\cdots,\ch_n)
\colon (\mathbf{r}_{\ph^{-1}(1)},\cdots,\mathbf{r}_{\ph^{-1}(n)})\rtarr 
(\mathbf{s_1},\cdots,\mathbf{s_n}).  \]
Here either $\ph^{-1}(j)$ is a single element $i$ or it is empty, in
which case $\mathbf{r}_{\ph^{-1}(n)} = \mathbf{1}$.  We interpolate an
analogous definition that is a follow--up to Remarks \ref{PN} and \ref{Q}.  
It will play an important role in our theory.

\begin{defn}  Let $\UP\int\UP$ denote the subcategory of $\PI\int\PI$ obtained 
by restricting all morphisms to be in $\UP$, thus using only projections.  Similarly, 
define $\UP\int \sD$ and $\sK\int\UP$ exactly as in the previous definition, 
but starting from the actions of $\UP$ on $\sD$ 
and $\sK$ on $\UP$ that are obtained by restricting the specifications
of Definition \ref{CLhatact} from $\PI$ to $\UP$. 
\end{defn}

The following observation helps analyze the structure of $\DK$.

\begin{lem}\label{contain} There are inclusions of categories
\begin{center}
$\sD\subset \UP\int \sD \subset \PI\int \sD \subset \DK \supset \sK\int\PI \supset \sK\int\UP\supset \sK.$\\
\end{center}
For maps
$(g;\ch)\colon (n;S)\rtarr (p;T)$ in $\sK\int\PI$ and 
$(\ph;d)\colon (m;R)\rtarr (n;S)$ in $\PI\int \sC$,
\[ (g;\ch)\com (\ph;d) = (1;\ch\com \la(g)(d))\com (g\ph;\si(g,\ph)). \]
The subcategories $\UP\int\sD$ and $\sK\int \UP$ generate $\DK$ under composition.
\end{lem}
\begin{proof} All but the first and last inclusions are obvious. The first inclusion sends an 
object $\mathbf{n}$ to $(1;\mathbf{n})$ and a morphism $d$ to $(\id;d)$.  The last sends an object 
$\mathbf{n}$ to $(n;\mathbf{1}^n)$ and a morphism 
$f\colon \mathbf{m}\rtarr \mathbf{n}$ to $(f;\id^n)$.  As noted in 
\cite[1.6]{Mult}, the displayed formula is obtained by composing the
legs of the following commutative diagram, where, for $1\leq j\leq n$ and 
$1\leq k\leq p$,
\[  \mathbf{r'_j} = \mathbf{r_{\ph^{-1}(j)}}, \ \ \ 
\mathbf{r''_k} = \sma_{\ps(j) = k}\mathbf{r_{\ph^{-1}(j)}}, \ \ \
\mathbf{s'_k} = \sma_{\ps(j) = k}\mathbf{s_j}.  \]
\[\xymatrix{
& & (p;R'') \ar[dr]^{(\text{id};\la(g)(c))} & &  \\
& (n;R') \ar[ru]^{(g;\text{id})} \ar[rd]^{(\text{id};c)} & & (p;S') 
\ar[rd]^{(\text{id};\ch)} & \\
(m;R) \ar[rr]_-{(\ph;c)} \ar[ur]^{(\ph;\text{id})}& &(n;S) \ar[rr]_-{(g,\ch)} \ar[ur]^{(g;\text{id})} & & (p;T) \\}  \]
Any morphism $(f;d)\colon (m;R)\rtarr (n;S)$ 
factors as the composite
\[\xymatrix@1{
(m;R)\ar[rr]^-{(f;id^n)} & & (n;R') \ar[rr]^-{(\id;d)} & & (n;S) \\} \]
where, with $\ph = \epz(f)$, $\mathbf{r'_j} = \sma_{\ph(i)=j}\mathbf{r_i}$.
This proves the last statement.
\end{proof}

With these constructions on hand, we define a category of ring operators
in analogy with our definition of a category of operators.  While our 
interest is in the case $\sJ=\sK\int\sD$, the general concept is convenient
conceptually.  For an injection $(\ph;\ch)\colon (m;R)\rtarr (n;S)$ in 
$\PI\int\PI$, define $\SI(\ph,\ch)$ to be the group of automorphisms 
$(\si;\ta)\colon (n;S)\rtarr (n;S)$ such that
$(\si;\ta)\text{Im}(\ph;\ta)\subset \text{Im}(\ph;\ta)$.

\begin{defn}\label{ringcatoper} A category of ring operators is a 
topological category $\sJ$ with objects those of $\PI\int\PI$
such that the inclusion $\PI\int\PI \rtarr \sF\int\sF$ factors as the
composite of an inclusion $\PI\int\PI\subset \sJ$ and a surjection
$\epz\colon \sJ\rtarr \sF\int\sF$, both of which are the identity on 
objects.  We require the maps 
$\sJ((\ell;Q),(m;R))\rtarr \sJ((\ell;Q),(n,S))$ induced by an 
injection $(\ph,\ch)\colon (m;R)\rtarr (n;S)$ in $\PI\int \PI$ to be
$\SI(\ph;\ch)$-cofibrations.  A map $\nu\colon \sI\rtarr \sJ$ of 
categories of operators is a continuous functor $\nu$ over $\sF\int\sF$ and 
under $\PI\int\PI$. It is an equivalence if each map
$\nu\colon \sI((m;R),(n;S)) \rtarr \sJ((m;R),(n;S))$
is an equivalence.
\end{defn}

When $\sJ =\hat{\sG}\int\hat{\sC}$ for an operad pair $(\sC,\sG)$, 
the cofibration condition is automatically satisfied since the maps 
in question are inclusions of components in disjoint unions. In fact,
with our new choice of details, the cofibration condition is not actually 
needed for the theory here.  So far, we have been following \cite{Mult}, 
but we now diverge and things begin to simplify.  We define $\sJ$-spaces 
without cofibration conditions and we ignore basepoints, which take care 
of themselves.  

\begin{defn}\label{Jspac} Let $\sJ$ be a category of ring operators.  
A $\sJ$-space
in $\sU$ is a continuous functor $Z\colon \sJ\rtarr \sU$, written 
$(n;S)\mapsto Z(n;S)$.  It is reduced if $Z(0;\ast)$ and 
$Z(1;0)$ are single points.  It is semi-special if the first two of
the following four conditions hold, and it is special if all four
conditions hold. 
\begin{enumerate}[(i)]
\item $Z(0;\ast)$ is aspherical.
\item The maps $\de''\colon 
Z(n;S)\rtarr \prod_{1\leq j\leq n}\, Z(1;\mathbf{s_j})$ with
coordinates induced by $(\de_j;\id)$ are equivalences.
\item $Z(1;0)$ is aspherical.
\item The maps $\de'\colon Z(1;\mathbf{s})\rtarr 
Z(1,\mathbf{1})^s$ 
with coordinates induced by $(1;\de_j)$ are equivalences. 
\end{enumerate}
It is very special if, further, the rig $\pi_0(Z(1;\mathbf{1}))$ is 
a ring.  A map $Z\rtarr W$ of $\sJ$-spaces is an equivalence if 
each $Z(n;S)\rtarr W(n;S)$ is an equivalence.
\end{defn}

\begin{rem}\label{subres} When $\sJ = \hat{\sG}\int \hat{\sC}$ for an 
operad pair $(\sC,\sG)$, the restriction of a $\sJ$-space $Z$ to the subcategory 
$\hat{\sC}$ of $\sJ$ is a $\hat{\sC}$-space $Z_{\oplus}$ and the restriction 
of $Z$ to the subcategory $\hat{\sG}$ is a $\hat{\sG}$-space $Z_{\otimes}$.
\end{rem}

\begin{defn}\label{Jsp2} Let $\sJ[\sU]$ denote the category of $\sJ$-spaces 
in $\sU$.
\end{defn}

Except for the ``very special'' notion,  Definitions \ref{Jspac} 
and \ref{Jsp2} apply
equally well if we relax our requirements on $\sJ$ to only require 
$\UP\int\UP$, rather than $\PI\int\PI$, to be contained in $\sJ$ and 
do not require the map from $\sJ$ to $\sF\int\sF$ to be a surjection.  
This leads us to our new choice of ground category for the multiplicative theory.

\begin{defn}  A $(\UP\int\UP)$-space is a functor $\UP\int\UP\rtarr \sU$, 
and we write $\sW$ for the category of $(\UP\int\UP)$-spaces. Changing 
notations from the additive theory, for a space $X$ we let $RX$ denote 
the $(\UP\int\UP)$-space that sends $(0;\ast)$ to a point and sends 
$(n;S)$ to $X^{s_1 \cdots s_n}$ for $n\geq 1$.
\end{defn}

Now comparisons of definitions give the following basic results, which
are \cite[2.4 and 2.6]{Mult}, where more details may be found. 

\begin{prop}\label{FFX} Let $\sJ = \sF\int\sF$.  Then the functor
$R\colon \sU\rtarr \sW$ embeds the category of commutative rig spaces
$X$ in the category of $\sJ$-spaces as the full subcategory of objects
of the form $RX$. 
\end{prop}
\begin{proof}[Sketch proof]  For a $\sJ$-space $RX$, the maps 
induced by $(\id;\ph_2)\colon (1;\mathbf{2})\rtarr (1;\mathbf{1})$ and 
$(\ph_2;\id)\colon (2;\mathbf{1}^2)\rtarr (1;\mathbf{1})$ give the 
addition and multiplication $X\times X\rtarr X$. The elements $0\in X$ 
and $1\in X$ are induced by the injections 
$(0;\ast)\colon (0;\ast)\rtarr (1;\mathbf{1})$ and 
$(0;\id)\colon (0;\ast)\rtarr (1;\mathbf{1})$ 
in $\PI\int\PI$. There is a unique way to extend a given
$(\sN,\sN)$-structure on $X$ to an action of $\sJ$ on $RX$.
\end{proof}

This result means that an $(\sF\int\sF)$-space structure on $RX$ is determined 
by its restriction to a commutative rig space structure on $X$.

\begin{prop} \label{JRX} Let $\sJ = \hat{\sG}\int\hat{\sC}$
for an operad pair $(\sC,\sG)$. Then the functor $R\colon \sU\rtarr \sW$ 
embeds the category of $(\sC,\sG)$-spaces $X$ in the category of 
$\sJ$-spaces as the full subcategory of objects of the form $RX$. 
\end{prop}
\begin{proof}[Sketch proof]  The restriction of an action of $\sJ$ on 
$RX$ to the operads $\sC$ and $\sG$ embedded in the subcategories 
$\hat{\sC}$ and $\hat{\sG}$ give the additive and multiplicative operad 
actions on $X$.   There is a unique way to extend a $(\sC,\sG)$-structure 
on $X$ to an action of $\sJ$ on $RX$. 
\end{proof}

Again, this means that a $(\hat{\sG}\int\hat{\sC})$-space structure on $RX$ 
is determined 
by its restriction to a $(\sC,\sG)$-space structure on $X$. 

By the same proof as those of Theorems \ref{Dchange} and \ref{Dchange2}, 
we have the following result.  It shows in particular that the homotopy
category of special $(\sF\int\sF)$-spaces is equivalent to the homotopy
category of special $(\hat{\sG}\int\hat{\sC})$-spaces, where $(\sC,\sG)$ is 
the canonical $E_{\infty}$ operad pair of \cite[\S3]{Prequel}.

\begin{thm}\label{Dchange3}  Let $\nu\colon \sI\rtarr \sJ$ be an
equivalence of categories of operators.  When restricted to the full
subcategories of special objects, the pullback of action
functor $\nu^*\colon \sJ[\sU]\rtarr \sI[\sU]$ induces an equivalence
of homotopy categories.
\end{thm}

\section{The definition of $(\hat{\sC},\hat{\sG})$-spaces}\label{MT2too}

Recall that we are writing $\sV$ for the category of $\UP$-spaces and $\sW$
for the category of $(\UP\int\UP)$-spaces. We need a pair of adjunctions analogous 
to the adjunction relating $\sU$ and $\sV$ (originally denoted $(L,R)$) that was used to compare monads in the additive theory.  Recall that we are now writing $R$ for the evident functor $\sU\rtarr \sW$. We can factor $R$ through $\sV$.

\begin{defn}\label{RLRL} 
For a space $X$, write $R'X = \{X^n\}$ for the associated
$\UP$-space.  For a $\UP$-space $Y$, let $R''Y$ be the $(\UP\int\UP)$-space 
that sends $(0;\ast)$ to a point (the empty product) and sends $(n;S)$ to 
$Y_{s_1}\times \cdots \times Y_{s_n}$ for $n>0$. Note that $RX = R''R'X$.  
Let $L'Y$ be the space $Y_1$ (previously denoted $LY$).  For an 
$(\UP\int\UP)$-space $Z$, let $L''Z$ be the $\UP$-space given by the spaces 
$Z(1;s)$, $s\geq 0$, and let $LZ = L'L''Z = Z(1;1)$.
\end{defn} 

It is easy to see what these functors must do on morphisms. Some
details are given in \cite[4.1]{Mult}, but the ``adjunctions'' 
claimed in that result are in fact not adjunctions because of 
basepoint and injection problems analogous to the mistake pointed 
out in \S1.  The following result is an elementary unbased substitute.
Its proof relies only on the universal property of cartesian products.

\begin{lem}\label{RLRL2} The following diagram displays two adjoint 
pairs of 
functors and their composite.
\[ \xymatrix{
\sU \ar@<1ex>[r]^-{R'} \ar@/^2.2pc/[rr]^-{R} 
& \sV \ar@<1ex>[r]^-{R''} \ar@<1ex>[l]^-{L'}
& \sW  \ar@<1ex>[l]^-{L''} \ar@/^2.2pc/[ll]^-{L} \\} \]
\end{lem}

Now let $(\sC,\sG)$ be an operad pair and abbreviate 
$\sJ=\hat{\sG}\int\hat{\sC}$.  
Proposition \ref{JRX} suggests the following definition of the intermediate category mentioned
in the introduction.\footnote{In \cite{Mult},
$(\hat{\sC},\hat{\sG})$-spaces were called $(\hat{\sC},\sG)$-spaces 
to emphasize the partial use of actual products implicit in 
their definition. I now feel that the earlier notation gives a 
misleading perspective.}

\begin{defn}\label{hatChatG} Let $\sJ = \hat{\sG}\int\hat{\sC}$
for an operad pair $(\sC,\sG)$. A $(\hat{\sC},\hat{\sG})$-space 
is an object $Y\in\sV$ together with a $\sJ$-space structure on
$R''Y$. It is special if $Y$ is special.  A map $f\colon Y\rtarr Y'$ 
of $(\hat{\sC},\hat{\sG})$-spaces is a map in $\sV$ such that $R''f$ 
is a map of $\sJ$-spaces. Thus, by definition, the functor $R''\colon \sV\rtarr \sW$ 
embeds the category of $(\hat{\sC},\hat{\sG})$-spaces as the full subcategory of
$\sJ$-spaces of the form $R''Y$. 
\end{defn}

\section{The monad $\bar{J}$ associated to the category $\sJ$}\label{barJstruc}

To compare $(\hat{\sC},\hat{\sG})$-spaces to $\sJ$-spaces on the one hand
and to $(\sC,\sG)$-spaces on the other, we must first analyze the monad associated to a category of ring operators.

\begin{defn}\label{Jss} Let $\bar{J}$ denote the monad on the category $\sW$ 
such that the category of ${\sJ}$-spaces is isomorphic to the category of 
$\bar{J}$-algebras in $\sW$. Define functors
$\tilde{J}\colon \sV\rtarr \sV$ and $J\colon \sU\rtarr \sU$ by
$\tilde{J} = L''\bar{J} R''$ and
$J = L'\tilde{J}R' = L\bar{J} R$.
\end{defn}

The construction of $\bar{J}$ is a special case of Construction \ref{construct}. 
Ignoring the monadic structure maps, we must find an explicit 
description of the functor $\bar{J}$ in order to relate it to the 
adjunctions of Lemma  \ref{RLRL2}.  This is where the main simplification 
of \cite{Mult} occurs.  We need some notations.\footnote{The details to follow come from \cite[\S7]{Mult}, but the combinatorial mistakes related to injections 
in $\PI\int\PI$ that begin in \cite[7.1(ii)]{Mult} have been circumvented
by avoiding basepoint identifications.}  Recall the description of 
$\hat{\sC}$ from Lemmas \ref{describe} and \ref{5.5}.

\begin{rems}
Observe that an ordered effective morphism 
$\ph\colon \mathbf{m}\rtarr \mathbf{n}$ 
in $\sF$ decomposes uniquely as $\ph=\ph_{m_1}\wed\cdots \wed \ph_{m_n}$, where
$m_j = |\ph^{-1}(j)|$ and $m_1+\cdots + m_j = m$.  Such $\ph$ determine and
are determined by partitions $M=(m_1,\cdots,m_n)$ of $m$.  In turn, for an object $(m;R)$ of 
$\UP\int\UP$, such a partition $M$ determines a partition of
$R = (\mathbf{r_1},\cdots, \mathbf{r_m})$ into  $n$ blocks, 
$R = (R_1,\cdots,R_n)$, where $R_j$ is the $j^{th}$ block subsequence of 
$m_j$ entries.  When $m=0$, we have a unique (ordered) effective morphism
$0\colon \mathbf{0}\rtarr \mathbf{n}$, a unique empty partition $M$ of 
$\mathbf{0}$, and a unique empty sequence $R$.   There are no effective
morphisms $\mathbf{m}\rtarr \mathbf{0}$ when $m>0$.
\end{rems} 

\begin{notns}\label{mRnots} Consider an object $(m;R)$ of $\UP\int\UP$, where 
$m\geq 0$ and $R = (\mathbf{r_1},\cdots, \mathbf{r_m})$ with each $r_i\geq 0$. In part (i), we use this notation but think of $(m;R)$ as $(m_j;R_j)$ where 
$1\leq j\leq n$. 
\begin{enumerate}[(i)]
\item  Fix $s\geq 0$. Say that a morphism 
$\ch\colon \sma_{1\leq i\leq m} \mathbf{r_i}\rtarr \mathbf{s}$ in $\sF$
is $R$-effective if for every $h$, $1\leq h\leq m$, and  
every $q$, $1\leq q\leq r_h$, there is a sequence $Q = (q_1,\cdots, q_m)$
in which $1\leq q_i\leq r_i$ for $1\leq i\leq m$ such that $q_h = q$ and 
$\ch(Q) \neq 0$. Let $\sE(R;s)$ denote the set of $R$-effective morphisms 
$\ch$, and define
\[ \sC(R;s) = \coprod_{\ch\in \sE(R;s)}\ 
\prod_{1\leq t\leq s}\sC(|\ch^{-1}(t)|). \]
Further, define $\SI(m;R)$ to be the group of automorphisms of $(m;R)$ in
$\UP\int\UP$.  
\item Fix $S=(\mathbf{s_1},\cdots,\mathbf{s_n})$, where $s_j\geq 0$.  
For a partition $M=(m_1,\cdots,m_n)$ of $m$ with derived partition
$R = (R_1,\cdots,R_n)$ of $R$, define
\[\sJ(M;R,S) = \prod_{1\leq j\leq n} \sG(m_j)\times \sC(R_j;s_j). \]
Further, define $\SI(M;R) = \prod_{1\leq j\leq n} \SI(m_j;R_j)\subset \SI(m;R)$.
\end{enumerate}
\end{notns}

\begin{rem}\label{gristle}  We clarify some special cases.  When $s=0$ in 
(i) and when $n=0$ in (ii), empty products of spaces are interpreted to be a single point. If $m=0$ in (i), the smash product over the empty sequence $R$ is interpreted as $\mathbf{1}$ and we allow $\ch$ to be 
$0\colon \mathbf{1}\rtarr \mathbf{0}$ or any injection 
$\mathbf{1}\rtarr \mathbf{s}$ in $\sF$.  If $m>0$ and any one 
$r_i=0$, then $R$-effectiveness forces all $r_i=0$ and we allow
$\ch = 0\colon \mathbf{0}\rtarr \mathbf{s}$.
\end{rem}

\begin{rem}\label{whistle} For later reference, we record when an
$R$-effective map $\ch$ in (i) can be in $\UP\subset \sF$ in the cases 
$s=0$ and $s=1$.  When $s=0$, we can only have $m=0$ and 
$\ch=0\colon \mathbf{1}\rtarr \mathbf{0}$ or $m>0$, all $r_i=0$, and 
$\ch = \text{id}= 0\colon \mathbf{0}\rtarr \mathbf{0}$.
When $s=1$, we can only have $m=0$ and 
$\ch = \text{id}\colon \mathbf{1}\rtarr \mathbf{1}$ or $m>0$, all $r_i=1$, 
and $\ch=\text{id}\colon \mathbf{1}\rtarr \mathbf{1}$.
\end{rem}

\begin{prop}\label{strucbarJ}  Let $Z\in \sW$.  Then 
$(\bar{J}Z)(0;\ast) = Z(0;\ast)$ and,
for $n>0$ and $S=(\mathbf{s_1},\cdots,\mathbf{s_n})$,
\[ (\bar{J}Z)(n;S) = \coprod_{(M;R)} \sJ(M;R,S)\times_{\SI(M,R)} Z(m;R), \]
where the union runs over all partitions $M=(m_1,\cdots,m_n)$ of all 
$m\geq 0$ and all sequences $R = (\mathbf{r_1},\cdots, \mathbf{r_m})$.
\end{prop}
\begin{proof}  We prove this by extracting correct details 
from \cite[\S7]{Mult}.  To begin with, observe that if 
$\ch'\colon \sma_{1\leq i\leq m}\mathbf{r'_i}\rtarr \mathbf{s}$ is a
map in $\sF$ that is not $R'$-effective, then it is a composite 
$\ch\com \sma_{1\leq i\leq m}\om_i$ where $\om_i\colon \mathbf{r'_i}
\rtarr \mathbf{r_i}$ is a projection and 
$\ch$ is 
$R$-effective. Indeed, suppose that $\ch'(Q)=0$ for all sequences $Q$ 
with $h^{th}$ term $q$, where $1\leq q\leq r_h$. Then 
$\ch' = (\ch'\com\sma_{1\leq i\leq m}\si_i)\com \sma_{1\leq i\leq m}\nu_i$, where $\nu_i = \si_i = \id\colon \mathbf{r_i'}\rtarr \mathbf{r_i'}$ 
for $i\neq h$, $\nu_h\colon \mathbf{r_h'}\rtarr \mathbf{r_h'-1}$
is the projection that sends $q$ to $0$ and is otherwise ordered, 
and $\si_h\colon \mathbf{r_h'-1}\rtarr \mathbf{r_h'}$ is the
ordered injection that misses $q$.  The required factorization is
obtained by repeating this construction inductively.

By Construction \ref{construct} and Definitions \ref{opercat} and  
\ref{wreath}, $(\bar{J}Z)(n;S)$ is a quotient of 
\[  \coprod_{(m;R)}\ \coprod_{(\ph;\ch)}\  \prod_{1\leq j\leq n}\ 
(\, \sG(|\ph^{-1}(j)|)\times \prod_{1\leq t\leq s_j}\sC(|\ch^{-1}_j(t)|)\, )
\times Z(m;R), \]
where $(\ph;\ch)$ runs over the morphisms $(m;R)\rtarr (n;S)$ in $\sF\int\sF$,
which means that $\ph\in\sF(\mathbf{m},\mathbf{n})$ and 
$\ch = (\ch_1,\cdots,\ch_n)$, where 
$\ch_j\in \sF(\sma_{\ph(i)=j}\mathbf{r_i},\mathbf{s_j})$.
The quotient is obtained using identifications that are induced by the morphisms of 
$\UP\int\UP$, namely the projections, which we think of as composites of proper 
projections and permutations.

In the description just given, we may restrict attention to those 
$(\ph;\ch)$ such that $\ph = \ph_{m_1}\wed\cdots\wed \ph_{m_n}$ 
for some partition $M$
of $m$ and $\ch = (\ch_1,\cdots,\ch_n)$, where $\ch_j$ is $R_j$-effective.
Indeed, if $(\ph';\ch')$ is not of this form, it factors as 
$(\ph;\ch)(\ps;\om)$ where $(\ph;\ch)$ is of this form and $\ps$ and
the coordinates of $\om$ are projections.  To construct $\ps$ and $\om$,
we use the observation above and record which elements other than $0$ of 
the sets $\mathbf{m}$ and the $\mathbf{r_j}$ are sent to $0$ by $\ph'$ and 
the $\ch_j'$.  Then $|\ph^{-1}(j)| = |(\ph')^{-1}(j)|$ for $1\leq j\leq n$,
$|\ch^{-1}(t)| = |(\ch')^{-1}(t)|$ for $1\leq t\leq s_j$, and any
morphism $(g';c')\colon (m';R')\rtarr (n;S)$ in $\sJ$ such that 
$\epz(g';c') = (\ph';\ch')$ factors as $(g;c)(\ps;\om)$ for some
morphism $(g;c)$ such that $\epz(g;c) = (\ph;\ch)$.   Up to permutations,
$(g;c) = (g';c')$ as elements of 
\[  \prod_{1\leq j\leq n}\ \sG(|\ph^{-1}(j)|)
\times \prod_{1\leq t\leq s_j}\sC(|\ch^{-1}(t)|). \]
This reduction takes account of the identifications defined using 
proper projections but ignoring permutations; the identifications defined using 
permutations are taken account of by passage to orbits over the $\SI(M;R)$.    
\end{proof}

Specializing $(n;S)$ to $(1;\mathbf{s})$ and then specializing $(1;\mathbf{s})$ 
to $(1;\mathbf{1})$ we obtain the following 
descriptions of the functors $\tilde{J} = L''\bar{J}R''$ and 
$J = L'\tilde{J}R'$. 

\begin{cor}\label{struchatJ}  Let $Y\in \sV$ and $X\in\sU$.  Then
\[   (\tilde{J}Y)_s = 
\coprod_{(m;R)} (\,\sG(m)\times \sC(R;s)\,)\times_{\SI(m;R)}
Y_{r_1}\times\cdots \times Y_{r_m}\]
and $JX$ is obtained by setting $s=1$ and replacing $Y_{r}$ by $X^{r}$. 
\end{cor}

The passage to orbits in Proposition \ref{strucbarJ} is well-behaved by the 
following observation.  It is \cite[7.4]{Mult}, and the proof is a straightforward inspection.

\begin{lem} Assume that $\sC$ and $\sG$ are $\SI$-free. Then the
action of $\SI(m;R)$ on $\sG(m)\times \sC(R;s)$ is free.  Therefore
the action of $\SI(M;R)$ on $\sJ(M;R,S)$ is free. 
\end{lem}  

This implies the following analogue of Lemma \ref{5.6}. 

\begin{prop} \label{5.6too} Assume that $\sC$ and $\sG$ are $\SI$-free.  
If $f\colon Z\rtarr Z'$ is an equivalence of $(\UP\int\UP)$-spaces, 
then so is $\bar{J}f$. Therefore, if $f\colon Y\rtarr Y'$ is an
equivalence of $\UP$-spaces, then so is $\tilde{J}f$, and if 
$f\colon X\rtarr X$ is an equivalence of spaces, then so is 
$Jf\colon JX\rtarr JY$. 
\end{prop}

\section{The comparison of $(\hat{\sC},\hat{\sG})$-spaces 
and $\sJ$-spaces}\label{Bonzai}

We can now compare $(\hat{\sC},\hat{\sG})$-spaces 
and $\sJ$-spaces by mimicking the comparison of $\sC$-spaces with 
$\hat{\sC}$-spaces given in Lemmas \ref{5.7} and \ref{5.8} and Theorem \ref{ChatC}.  
We need three preliminary results.

\begin{prop}\label{5.7too}  Let $Y\in \sV$. Then the natural map
$$\de''\colon \bar{J}R''Y\rtarr R''L''\bar{J}R''Y \equiv R''\tilde{J} Y$$
is an isomorphism.  Therefore $\tilde{J}$ inherits a structure of monad from 
$\bar{J}$ and the functor $R''$ embeds the category of $\tilde{J}$-algebras as the full 
subcategory of the category of $\bar{J}$-algebras consisting of those 
$\bar{J}$-algebras of the form $R''Y$.  
\end{prop}
\begin{proof}
By our description of $\bar{J}$, we see that $(\bar{J}R''Y)(0;\ast)$
is a point and, for $n>0$,
\[ (\bar{J}R''Y)(n;S) 
= \coprod_{(M;R)} \prod_{1\leq j\leq n} 
(\,\sG(m_j)\times \sC(R_j;s_j)\,)\times_{\SI(m;R)}
Y_{r_1}\times\cdots \times Y_{r_m}. \]
On the other hand,
\[ (R''\tilde{J}Y)(n;S) 
= \prod_{1\leq j\leq n} \coprod_{(m;R)} 
(\,\sG(m)\times \sC(R;s_j)\,)\times_{\SI(m;R)}
Y_{r_1}\times\cdots \times Y_{r_m}. \]
The map $\de''$ gives the identification that is obtained
by commuting disjoint unions past cartesian products and assembling block partitions. 
By Proposition \ref{omniold}, the second statement is a formal consequence of the first. 
\end{proof}

We restate the second statement since it is pivotal to our later
comparison of $(\hat{\sC},\hat{\sG})$-spaces and $(\sC,\sG)$-spaces.

\begin{cor}\label{linch}  The categories of $(\hat{\sC},\hat{\sG})$-spaces
and $\tilde{J}$-algebras are isomorphic.
\end{cor}

There are other comparisons of functors that one might hope to make and
that fail.  We record some of them.  These failures dictate the conceptual 
outline of the theory.  They clarify why we must introduce the notion of a
semi-special $(\UP\int\UP)$-space and why we must use the intermediate 
category of $(\hat{\sC},\hat{\sG})$-spaces rather than compare 
$(\sC,\sG)$-spaces and $\sJ$-spaces directly. 

\begin{rem}\label{fail}  
Observe that $(\bar{J}Z)(1;\mathbf{s})$ depends 
on all $Z(m;R)$ and not just the $Z(1;\mathbf{s})$.  Therefore $L''\bar{J}Z$
is not isomorphic to $\tilde{J}L''Z$, in contrast to Lemma  \ref{5.7}.  Similarly, 
$(\tilde{J}Y)_1$ depends
on all $Y_s$ and not just $Y_1$.  Therefore $L'\tilde{J}Y$ is not
isomorphic to $JL'Y$.  Again, for a space $X$, $\tilde{J}R'X$ is not
isomorphic to $R'JX$. In fact, $(\tilde{J}R'X)_n$ is not even 
equivalent to $(JX)^n$.  Thus $\bar{J}Z$ need not be
special when $Z$ is special and $\tilde{J}Y$ need not be special when $Y$ 
is special. 
\end{rem}

\begin{prop}\label{5.8free} Assume that $\sC$ and $\sG$ are $\SI$-free.  
If $Z$ is a semi-special $(\UP\int\UP)$-space then so is $\bar{J}Z$, 
hence $\bar{J}$ restricts to a monad on the category of semi-special
$(\UP\int\UP)$-spaces.
\end{prop}
\begin{proof} 
Applying Proposition \ref{5.6too} to the horizontal arrows in the diagram
\[ \xymatrix{
(\bar{J}Z)(n;S) \ar[rr]^-{\bar{J}\de''} \ar[d]_{\de''} 
& & (\bar{J}R'' L'' Z)(n;S) \ar[d]_{\iso}^{\de''} \\
(R''L''\bar{J}Z)(n;S) 
\ar[rr]_-{R''L''\bar{J}\de''}
& & (R''L''\bar{J}R''L'' Z)(n;S),\\} \]
we see that its left vertical arrow is an equivalence.
\end{proof}

As promised, we can now compare $(\hat{\sC},\hat{\sG})$-spaces in $\sV$ to 
$\bar{J}$-spaces in $\sW$ by simply repeating the proof of Theorem \ref{ChatC}. We again use the two-sided monadic bar construction of \cite{Geo} together with 
the monadic 
generalities in Appendix A (\S14), general properties of geometric realization, and 
the following analogue of Lemma \ref{geohat}, whose proof is just like
that of \cite[12.2]{Geo}.

\begin{lem}\label{geobar} For simplicial objects $Z$ in the category $\sW$, there is a natural 
isomorphism $\nu\colon |\bar{J}Z|\rtarr \bar{J}|Z|$ such that the
following diagrams commute.
\[\xymatrix{
|Z| \ar[r]^-{|\et|} \ar[dr]_{\et} & |\bar{J}Z| \ar[d]^{\nu}\\
 & \bar{J}|Z| \\ }
\ \ \  \text{and}\ \  \
\xymatrix{
|\bar{J}\bar{J}Z| \ar[r]^-{|\mu|} \ar[d]_{\bar{J}\nu\com\nu} 
& |\bar{J}Z | \ar[d]^{\nu} \\
\bar{J}\bar{J}|Z| \ar[r]_{\mu} &  \bar{J}Z \\ }  \]
If $(Z,\xi)$ is a simplicial $\bar{J}$-algebra, then $(|Z|,|\xi|\com\nu^{-1})$
is a $\bar{J}$-algebra.
\end{lem}

\begin{thm} If $\sC$ and $\sG$ are $\SI$-free, then the functor 
$R''\colon \sV\rtarr \sW$ induces an equivalence from the homotopy category 
of special $(\hat{\sC},\hat{\sG})$-spaces to the homotopy category of special 
$\sJ$-spaces.
\end{thm}
\begin{proof} We repeat the proof of Theorem \ref{ChatC}.  Again, Proposition \ref{5.7too} puts us 
into one of the two contexts discussed in general terms in Proposition \ref{omniold}.  
Let $Y$ be
a $(\hat{\sC},\hat{\sG})$-space and $Z$ be a ${\sJ}$-space. By Proposition \ref{omniold}, 
$\tilde{J}L''$ is a $\bar{J}$-functor, and we can define 
a functor $\LA''\colon \bar{J}[\sW]\rtarr \tilde{J}[\sV]$ by sending a 
$\bar{J}$-algebra $Z$ to the $\tilde{J}$-algebra
\[ \LA'' Z = B(\tilde{J}L'', \bar{J}, Z). \]
By Corollaries \ref{coromni1} and \ref{coromni2}, together with general 
properties of the geometric realization of simplicial spaces proven in 
\cite{Geo}, we have a diagram
\[ \xymatrix@1{
Z & B(\bar{J},\bar{J},Z)\ar[r]^-{\de''} \ar[l]_-{\epz} 
& B(R''\tilde{J}L'',\bar{J},Z)\iso R''\LA'' Z\\} \]
of $\bar{\sJ}$-spaces in which the map $\epz$ is a homotopy equivalence with natural 
homotopy inverse $\et$ and the map $\de'' = B(\de'',\id,\id)$ is an equivalence when $Z$ 
is semi-special.  Thus the diagram displays a natural weak equivalence between 
$Z$ and $R''\LA'' Z$.  When $Z=R''Y$, the displayed diagram is obtained by applying 
$R''$ to the analogous diagram
\[ \xymatrix@1{
Y & B(\tilde{J},\tilde{J},X)\ar[r]^-{\iso} \ar[l]_-{\epz} & 
B(\tilde{J}L'',\bar{J}, R''Y) = \LA'' R''Y\\} \]
of $\tilde{J}$-algebras, in which $\epz$ is a homotopy equivalence with natural
inverse $\et$. 
\end{proof}

\section{Some comparisons of monads}\label{addenda}

To clarify ideas and to set up the comparison of $(\sC,\sG)$-spaces and 
$(\hat{\sC},\hat{\sG})$-spaces, we define and compare several other monads 
and functors related to those already specified.  We again fix 
$\sJ =\hat{\sG}\int\hat{\sC}$ with associated monad $\bar{J}$ on $\sW$.  
Recall that $\tilde{J} =L''\bar{J} R''$ and $J = L'\tilde{J}R' = L\bar{J} R$. Taking $\sC$ or $\sG$ to be the operad $\sQ$ of Remark \ref{Q}, the following definition is a special case of Definition \ref{Jss}.

\begin{defn}\label{CGss} Let $\bar{C}$ denote the monad on $\sW$ whose algebras are
the $\UP\int \hat{\sC}$-spaces and let $\bar{G}$ denote the monad on $\sW$ whose algebras 
are the $\hat{\sG}\int \UP$-spaces.
\end{defn}

Similarly, the following result is a special case of Proposition \ref{5.7too}. 

\begin{prop}\label{5.7also} Let $Y\in \sV$. The natural maps
\[ \de''\colon \bar{C}R''Y\rtarr R''L''\bar{C} R'' Y
\ \ \text{and} \ \ \de''\colon \bar{G}R''Y\rtarr R''L''\bar{G} R''Y \]
are isomorphisms. Therefore the monad structures on $\bar{C}$
and $\bar{G}$ induce monad structures on $L''\bar{C} R''$ and $L''\bar{G} R''$
such that the functor $R''$ embeds the category of $L''\bar{C} R''$-algebras
$Y$ isomorphically onto the full subcategory of $\bar{C}$-algebras of 
the form $R''Y$ and embeds the category of $L''\bar{G} R''$-algebras $Y$ 
isomorphically onto the full subcategory of $\bar{G}$-algebras 
of the form $R''Y$. 
\end{prop}

\begin{prop}\label{keyident} 
The monad $L''\bar{C}R''$ can be identified with the monad $\hat{C}$.
\end{prop}
\begin{proof}
Inspection of the case $\sG = \sQ$ of Corollary \ref{struchatJ} makes 
clear that the underlying functors can be identified.  The structure
maps of the monads agree under the identifications since they are induced 
by the structure maps of the operad $\sC$.  
\end{proof}

The analogue for $\sG$ is not true, and we introduce an abbreviated notation.

\begin{defn}\label{Grudge} Define $\tilde{G}$ to be the monad 
$L''\bar{G}R''$ and let $\tilde{G}[\sV]$ denote the category of
$\tilde{G}$-algebras in $\sV$.
\end{defn}
 
We now appeal to Beck's results on distributivity and monads, which
are summarized in Theorem \ref{Beckthm} below.  We used monads in $\sV$ and 
$\sW$ to compare $(\hat{\sC},\hat{\sG})$-spaces to $\sJ$-spaces and we 
will use monads in $G[\sU]$ and $\tilde{G}[\sV]$ to compare 
$(\sC,\sG)$-spaces to $(\hat{\sC},\hat{\sG})$-spaces in the next
section. Beck's results will
allow us to complete that comparison conceptually.  

For any pair of monads, $C$ and $G$ say, on the same category $\sV$, there 
is a notion of an action of $G$ 
on $C$, spelled out in Definition \ref{CGact}.  When $G$ acts on $C$, $C$ restricts to a 
monad on $G[\sV]$.  As is made precise in Theorem \ref{Beckthm}, it is equivalent that $CG$ is a 
monad on $\sV$ such that $CG$-algebras in $\sV$ are the same as 
$C$-algebras in $G[\sV]$.  The shift in perspective that this allows is 
crucial to our intermediate use of $(\hat{\sC},\hat{\sG})$-spaces.  

As Theorem \ref{Beckthm} also makes precise, a third equivalent condition is
that $G$ acts on $C$ if and only if there 
is a natural map $\rh\colon GC\rtarr CG$ that makes appropriate diagrams commute. We agree to call such a map $\rh$ a distributivity map since it 
encodes distributivity data. We have three results that arise from this perspective and tie things together.  The first two will be given here
and the third in the next section.  It may be helpful to the reader if we 
first list the relevant endofunctors on our three ground categories.
\begin{equation}\label{inW} 
\bar{C},\ \ \bar{G},\ \ \bar{C}\bar{G},\ \ \bar{J} \ \ \text{on} \ \ \sW
\end{equation}
\begin{equation}\label{inV}
\hat{C},\ \ \tilde{G}, \ \ \hat{C}\tilde{G}, 
\ \ \tilde{J} \ \  \text{on} \ \ \sV 
\end{equation}
\begin{equation}\label{inU}
C, \ \  G, \ \  CG,\ \ J, \ \  \text{on} \ \ \sU. 
\end{equation}
All of these functors except $J$ are monads, as we shall see, and
inclusions of operad pairs induce a number of obvious maps between them.  
Our promised three results, one for each of $\sW$, $\sV$, and $\sU$,
show how these monads and maps are related.  

\begin{thm}\label{tooW} 
There is a distributivity map $\bar{\rh}\colon \bar{G}\bar{C}\rtarr \bar{C}\bar{G}$
which makes the following diagram commute.
\[ \xymatrix{
\bar{G}\bar{C}\ar[rr]^-{\bar{\rh}} \ar[d] & & \bar{C}\bar{G} \ar[d]\\
\bar{J}\bar{J} \ar[r]_-{\bar{\mu}} & \bar{J} & 
\bar{J}\bar{J} \ar[l]^-{\bar{\mu}}. \\} \]
The composite $\bar{G}\bar{C}\rtarr \bar{J}$ in the diagram is an isomorphism
of monads on $\sW$. 
\end{thm}
\begin{proof}[Sketch proof]  Modulo our variant monads, this is a version of
\cite[6.12]{Mult}, where more details can be found.  The diagram and the 
constructions of the monads dictate the definition of $\bar{\rh}$, and a
precise formula for the map is dictated by the commutation relation in
Lemma \ref{contain}.  Diagram chases show that $\rh$ satisfies the properties of a 
distributivity map specified in Theorem \ref{Beckthm}(iii).  It follows that
$\bar{G}$ acts on $\bar{C}$, so that $\bar{C}$ is a monad on $\bar{G}[\sW]$
and $\bar{C}\bar{G}$ is a monad on $\sW$ with the same algebras. 
The displayed diagram itself implies that its composite is a map of monads.
It is surjective because $\sG\int\UP$ and $\UP\int \sC$ generate $\sJ$ under 
composition, and inspection shows that it is injective.   More conceptually,
a $\bar{J}$-space $Z$ is a $\bar{C}\bar{G}$ by pullback and, conversely,
suitably compatible actions of $\sG\int\UP$ and $\UP\int\sC$ on $Z$ 
determine an action of $\sJ$ on $Z$.  This implies that a $\bar{C}\bar{G}$-algebra 
is the same thing as a $\bar{J}$-algebra, so that the two monads have the same algebras.  
In turn, by the monadicity of the forgetful functor from $\sJ$-spaces to $\sW$ 
(as in \cite[App.A (\S14)]{Prequel}), that implies that the map of monads
$\bar{C}\bar{G} \rtarr \bar{J}$ is an isomorphism.  
\end{proof}

Using Theorem \ref{Beckthm} together with Propositions \ref{5.7also}, and 
\ref{omniold}, we find that the following result, which is a
version of \cite[6.13]{Mult}, is a formal consequence of the 
previous one.

\begin{thm}\label{tooV} The composite displayed in the following diagram
is a distributivity map $\tilde{\rh}\colon \tilde{G}\hat{C} \rtarr \hat{C}\tilde {G}$.
\[ \xymatrix{
\tilde{G}\hat{C} = L''\bar{G}R''L''\bar{C}R'' \ar[d]_{\tilde{\rh}}
 \ar[rr]^-{(L''\bar{G}\de''\bar{C}R'')^{-1}} & &
L''\bar{G}\bar{C} R''
\ar[d]^{L''\bar{\rh} R''} \\
\hat{C}\tilde{G} = L''\bar{C}R''L''\bar{G}R'' & & 
L''\bar{C}\bar{G} R''
\ar[ll]_-{L''\bar{C}\de''\bar{G}R''}\\} \]
The natural composite 
\[ \hat{C}\tilde{G}\rtarr \tilde{J}\tilde{J}
\rtarr \tilde{J} \] 
is an isomorphism of monads $\hat{C}\tilde{G}\rtarr \tilde{J}$ on $\sV$.
\end{thm}

The following key result is now immediate from Corollary \ref{linch} and 
Theorem \ref{Beckthm}.

\begin{cor}\label{finally} The categories of $(\hat{\sC},\hat{\sG})$-spaces, of 
$\hat{C}\tilde{G}$-algebras in $\sV$, and of $\hat{C}$-algebras in $\tilde{G}[\sV]$ 
are isomorphic.
\end{cor}

\section{The comparison of $(\sC,\sG)$-spaces and 
$(\hat{\sC},\hat{\sG})$-spaces}\label{Mult02}

Again, let $(\sC,\sG)$ be an operad pair, so that $\sG$ acts on $\sC$. 
We have the monads $C$ and $G$ on $\sU$ of the prequel \cite[(4.1)]{Prequel}.  
As a reminder, recall that we have changed notations from there, so that 
$CX$ here is as specified in (\ref{CX}), and similarly for $GX$.
We have isomorphisms of categories $\sC[\sU]\iso C[\sU]$ and
$\sG[\sU]\iso G[\sU]$. In \cite{Prequel}, we gave a monadic description 
of $E_{\infty}$ ring spaces using monads that take account of the 
basepoint $0$ and its role as a zero for the multiplication.  Here 
we are ignoring basepoints and, using the language of Appendix B (\S15), 
we have the following alternative version of \cite[4.8]{Prequel}.  Again, the
difference is just a question of whether or not basepoints are thought
of as preassigned. 

\begin{prop}\label{CGmon} The monad $G$ on $\sU$ acts on the 
monad $C$, so that $C$ induces a monad, also denoted $C$, 
on the category $G[\sU]$ of $G$-algebras.  The category of
$(\sC,\sG)$-spaces is isomorphic to the category of 
$C$-algebras in $G[\sU]$.  
\end{prop} 
\begin{proof}
Taking \cite[1.4]{Prequel} into account, the proof is the same as 
that of \cite[4.8]{Prequel}, whose missing details (from \cite[VI\S1]{MQR}) 
can be read off directly from Definition \ref{CLact}.  
\end{proof}

Since $G$ acts on $C$, Theorem \ref{Beckthm} gives a corresponding distributivity 
map.  The following result, which combines versions of \cite[6.11 
and 6.13]{Mult}, describes it.

\begin{thm}\label{tooU} The distributivity map $\rh\colon GC\rtarr CG$ 
is the composite of the maps $\rh_1$ and $\rh_2$ defined by the commutativity 
of the upper and lower rectangles in the following diagram.
\[ \xymatrix{
GC = L'\tilde{G}R'L'\hat{C}R' \ar[d]_{\rh_1} 
\ar[rr]^-{(L'\tilde{G}{\de}'\hat{C}R')^{-1}}& & 
L'\tilde{G}\hat{C} R' \ar[d]^{L'\tilde{\rh}R'} \\
J = L'\tilde{J} R' \ar[d]_{\rh_2}& & 
L'\hat{C}\tilde{G}R' \ar@{<->}[ll]_-{\iso} \ar[d]^{L'\hat{C}\de'\tilde{G}R'}\\
CG \ar@{=}[rr] & & L'\hat{C} R'L' \tilde{G} R' \\} \]
\end{thm}

Thinking of $(\sC,\sG)$-spaces as multiplicatively enriched $\sC$-spaces, 
we have in effect changed ground categories from $\sU$ to $G[\sU]$.  Since 
$\hat{\sG}$ acts on $\hat{\sC}$, as explained in \S\ref{Mult0}, 
one might well expect the monad $\hat{G}$ to act on the monad $\hat{C}$, 
but that 
is false.\footnote{Vigleik Angeltveit showed me convincingly exactly
how this fails, and he pointed out some faulty details in a purported 
description of the monad $\tilde{G}$ given in an earlier
draft of this paper.}  However, as we saw in the previous section,
the monad $\tilde{G}$ does act on $\hat{C}$.  We could extract an
explicit description of $\tilde{G}$ by specializing the explicit
description of $L''\bar{J}R''$ given in Corollary \ref{struchatJ} and using
the generalization of Remark \ref{whistle} to $s>1$.  We omit the details
since we have no need for them.  However, we observe that Remark \ref{whistle} 
implies the following result. 

\begin{lem}\label{surprise} For $Y\in \sV$, the space $(\tilde{G}Y)_0$ 
can be identified 
with $G(Y_0)$, and the space $L'Y = (\tilde{G}Y)_1$ can be identified 
with $G(Y_1)$.
\end{lem}

This implies the following analogue of Lemma  \ref{5.7}. Recall that
$(L,R)$ there is the same as $(L',R')$ here. 

\begin{lem}\label{5.7G} Let $X\in \sU$.  Then 
$(\tilde{G}R'X)_0 =G(\ast)$,
$L'\tilde{G} R'X \equiv (\tilde{G}R'X)_1 = GX$,  
and the natural map 
$L'\tilde{G}\de\colon L'\tilde{G}Y \rtarr L'\tilde{G}R'L'Y =GL'Y$ 
is an isomorphism. 
\end{lem}

Of the previous three results, only the last statement of Lemma \ref{5.7G} is
on the main line of development.
Returning to the desired comparison of $(\sC,\sG)$-spaces and
$(\hat{\sC},\hat{\sG})$-spaces, the following result puts us into 
the framework of Appendix A (\S14).

\begin{lem}\label{CuteAdj} The adjunction $(L',R')$ induces an 
adjunction
\[ G[\sU](L'Y,X)\iso \tilde{G}[\sV](Y,R'X). \]
\end{lem}
\begin{proof} It is obvious that $L'$ takes $\tilde{G}$-algebras $Y$ to
$G$-algebras since $L'\tilde{G}Y = GL'Y$ and we can restrict the action maps
accordingly.  We claim that $R'$ takes $G$-algebras $X$ to 
$\tilde{G}$-algebras $R'X$.  To see this conceptually, we can modify 
slightly the definition of an operad pair by allowing $\sC(0)$ to be
empty.  Then, quite trivially, any operad $\sG$ acts on our operad 
$\sQ$ such that $\hat{\sQ} = \UP$. Clearly, we can identify 
$(\sQ,\sG)$-spaces with $G$-algebras in $\sU$, and these can 
then be identified with $\hat{\sG}\int \hat{\sQ}$-spaces of the form 
$RX = R''R'X$, as in Proposition \ref{JRX}.  As in Definition \ref{hatChatG}, 
we define $(\hat{\sQ},\hat{\sG})$-spaces $Y$ to be $(\hat{\sG}\int \hat{\sQ})$-spaces 
of the form $R''Y$, and it is then obvious that $R'X$
is a $(\hat{\sQ},\hat{\sG})$-space.  Finally, as in Corollary \ref{finally},
we see that $(\hat{\sQ},\hat{\sG})$-spaces can be identified with
$\tilde{G}$ algebras in $\sV$.
\end{proof}

Since $(\hat{\sC},\hat{\sG})$-spaces are the same as $\hat{C}$-algebras
in the category of $\tilde{G}$-algebras, by Corollary \ref{finally}, Theorem \ref{ChatC} admits the following multiplicative elaboration.  In effect, we just change ground categories from 
$\sU$ and $\sV$ to $G[\sU]$ and $\tilde{G}[\sV]$.  Otherwise the proof
is exactly the same.

\begin{thm}\label{ChatCtootoo}  If $\sC$ is $\SI$-free, then the functor 
$R'\colon \sU\rtarr \sV$ induces an equivalence from the homotopy category 
of $(\sC,\sG)$-spaces to the homotopy category of special 
$(\hat{\sC},\hat{\sG})$-spaces.
\end{thm}

This gives the bottom right pair of parallel arrows in (\ref{DIAG}).

\section{Permutative categories in infinite loop space theory}\label{Per}

We assume familiarity with the notion of a symmetric monoidal category.
That is just a (topological) category $\sA$ with a product and a unit object 
which satisfy the associativity, commutativity, and unit laws up to coherent 
natural isomorphism.  If the associativity and unit laws hold strictly,
then $\sA$ is said to be permutative.\footnote{I believe that this pleasant and appropriate 
name is due to Don Anderson \cite{DWA}.}  There is no loss of generality in
restricting to permutative categories since any (small) symmetric monoidal
category is equivalent to a permutative category \cite{Isbell, MayPer}. 
One cannot also make the commutativity law hold strictly, and it is the lack 
of strict commutativity that leads to the higher homotopies implicit in 
infinite loop space theory.  Thus permutative categories are the strictest kind of 
symmetric monoidal category that one can define without loss of generality. 

Precisely, a permutative category $\sA$ has an associative 
product $\Box$ with strict two-sided unit object $u$ and a natural 
commutativity involution $c\colon A\Box B\rtarr B\Box A$ such that
$c=\text{id}\colon A=u\Box A\rtarr A\Box u = A$ and the following
diagram commutes.
\[ \xymatrix{
A\Box B\Box C \ar[rr]^-{c}  \ar[dr] _{\text{id}\Box c} 
&& C\Box A\Box B \\
& A\Box C\Box B \ar[ur]_{c\Box\text{id}} & \\} \]
More generally, rather than having a set of objects, $\sA$ might be an
internal category in $\sU$, so that it has a space of objects and 
continuous source, target, identity, and composition maps.

A functor $F\colon \sA\rtarr \sB$ between symmetric monoidal categories
is lax symmetric monoidal if there is a map
$\al\colon u_{\sB}\rtarr F(u_{\sA})$ and a natural transformation
\[ \PH\colon \Box_{\sB}\com F\times F \rtarr F\com\Box_{\sA} \]
of functors $\sA\times \sA\rtarr \sB$ satisfying appropriate 
coherence conditions.  An op-lax functor is defined similarly, 
but with maps going in the other direction. We say that $F$ is strong
(instead of lax or op-lax) if $\al$ and $\PH$ are isomorphisms and that
$F$ is strict if $\al$ and $\PH$ are identities.  The strict notion
is only interesting when $\sA$ and $\sB$ are permutative.

The relationship between permutative categories and spectra was  
axiomatized in \cite{MayPer2, Pair}.   An infinite loop space 
machine defined on the category $\sP\sC$
of permutative categories is a functor $\bE$ from $\sP\sC$ to any good category 
of spectra (say $\OM$-prespectra for simplicity) together with a natural group 
completion $\io\colon B\sA\rtarr {\bE}_0\sA$, where $B\sA$ is the classifying space 
of $\sA$. Up to natural equivalence, there is a unique such machine $(\bE,\io)$
\cite[Thm 3]{MayPer2}.  We have omitted the
specification of the morphisms of $\sP\sC$.  Strict morphisms were used in 
\cite{MayPer2}.  However, there is a functor from the category of permutative 
categories and lax morphisms to the category of permutative categories and strict 
morphisms that can be used to show that the uniqueness theorem remains valid when 
the morphisms in $\sP\sC$ are taken to be lax; see \cite[4.3]{Pair}.

There are several constructions of such a machine $(\bE,\io)$.  There is an
$E_{\infty}$ operad $\tilde{\SI}$ in $\sC\!at$ whose $j^{th}$ category $\tilde{\SI}_j$ is the 
translation category of the symmetric group $\SI_j$.  It was
defined in \cite[\S4]{MayPer} and, in more detail (and with a minor correction) in \cite[VI\S4]{MQR}.  
As observed in these sources, there are functors 
$\tilde{\SI}_j\times \sA^j\rtarr \sA$ that specify an action of $\tilde{\SI}$ 
on $\sA$.  Passing to classifying spaces, we have an action of the $E_{\infty}$ 
operad $\sD = B\tilde{\SI}$ of spaces on the space $B\sA$.  As recalled in
\cite[9.6]{Prequel}, $\sD$ is the topological version of the Barratt-Eccles
operad \cite{BE2}.  The additive infinite loop space machine of \cite{Geo}, 
as described in \cite[\S9]{Prequel}, gives the required machine $(\bE,\io)$.  

Alternatively, there are at least two ways, one combinatorial and the 
other conceptual, to construct a special $\sF$-category from a permutative category.  
Application of the classifying space functor then gives a
special $\sF$-space, to which Segal's infinite loop space machine 
\cite{Seg2} can be applied.  The combinatorial construction is due to
Segal \cite{Seg2}.  Full details are supplied in \cite[Construction 10]{MayPer2}.  It is essential to the uniqueness theorem there that the construction actually gives a functor from $\sP\sC$ to the category 
$\sF\sP\sC$ of special functors $\sF\rtarr \sP\sC$.  A defect of the 
construction is that it is functorial only on strict rather than lax
morphisms of permutative categories.  The conceptual construction is an 
application of  ``Street's first construction'' from \cite{Street} and is spelled out in \cite[\S\S3,4]{Pair}.  It does not give a functor to 
$\sF\sP\sC$, but it is functorial on lax morphisms.  We say a bit more
about it, or rather its bipermutative analogue, in the next section.   

While there is an essentially unique way to construct spectra from
permutative categories, there is another consistency statement that
is of considerable importance in some of the topological applications.  
In \cite[\S2]{Prequel},  we recalled the notion of a monoid-valued $\sI$-FCP 
(functor with cartesian product) from \cite[I\S1]{MQR} and the more 
modern source \cite[Ch. 23]{MS}.  As explained in those sources, such 
a functor $G$ can be extended from the category $\sI$ of finite dimensional inner product 
spaces to the category $\sI_c$ of countably infinite dimensional inner product spaces by 
passage to colimits.  Then $G(\bR^{\infty})$ is an 
$\sL$-space, where $\sL$ is the linear isometries $E_{\infty}$ operad, and
so is $BG \equiv BG(\bR^{\infty})$.  These can be fed into the additive infinite 
loop space machine of \cite[\S9]{Prequel}.  On the other hand, the 
$G(\bR^n)$ are the morphism spaces of a permutative category with object 
set $\{n|n\geq 0\}$ and no morphisms $m\rtarr n$ for $m\neq n$.  It is
proven in \cite{MayImon} that the spectrum obtained from the $\sL$-space
$BG$ is the connected cover of the spectrum obtained from the permutative
category $\amalg_{n\geq 0} G(\bR^n)$, whose $0^{th}$ space is equivalent
to $BG\times \bZ$. 

\section{What precisely are bipermutative categories?}\label{Biper}

We would like to assume familiarity with the notion of a symmetric bimonoidal category, but the categorical literature on this important topic is strangely meager.  Intuitively, we have a category $\sA$ with two symmetric monoidal
products, $\oplus$ and $\otimes$, with unit objects denoted $0$ and $1$. 
The distributivity laws must hold, at least up to coherent natural transformation.  
As usual, the notion of coherence has to be made precise
in order to have a sensible definition, and a coherence theorem is necessary 
for the notion to be made rigorous.  The only systematic study of coherence
and the only coherence theorems that I know of in this context are those of Laplaza \cite{Laplaza, Laplaza2}.  The essential starting point is to 
formulate distributivity precisely.  Laplaza requires a left\footnote{In algebra,
the left distributivity law states that $a(b+c) = ab+ac$, so that
left multiplication by $a$ is linear.  Curiously, \cite{EM} has left and
right reversed, viewing (\ref{dis}) as right distributivity.} distributivity {\it monomorphism}
\begin{equation}\label{dis}  
\de\colon A\otimes (B\oplus C) \rtarr (A\otimes B)\oplus (A\otimes C). 
\end{equation}
If we define $F_A(-) = A\otimes (-)$ and think of $(\sA,\oplus)$ as a 
symmetric monoidal category, then coherence says in part that $F_A$ is
an op-lax symmetric monoidal functor under $\de$ and the evident unit 
isomorphism.  Therefore, we might say that LaPlaza requires a 
semi op-lax distributivity law.  A fully op-lax distributivity law would 
delete the monomorphism requirement. A lax distributivity law would have 
the arrow point the other way.  In the interesting examples, $\de$ is a natural 
isomorphism, and I prefer to require that in the definition, as I did in 
\cite[p. 153]{MQR}.  Perhaps we should then call these strong symmetric 
bimonoidal categories.   
In any case, the left and right distributivity laws, $\de$ and $\de'$ say,  
must determine each other
by the following commutative diagram, in which $c_{\otimes}$ is the
commutativity isomorphism for $\otimes$. \begin{equation}\label{dis2}
\xymatrix{
A\otimes (B\oplus C) \ar[r]^-{\de} \ar[d]_{c_{\otimes}} ^{\iso}
& (A\otimes B)\oplus (A\otimes C) 
\ar[d]^{c_{\otimes}\oplus c_{\otimes}}_{\iso} \\
(B\oplus C)\otimes A\ar[r]_-{\de{'}} 
& (B\otimes A)\oplus (C\otimes A). \\} 
\end{equation}

As originally specified in \cite[VI\S3]{MQR}, bipermutative categories
give the strictest kind of strong symmetric bimonoidal category that 
one can define without loss of generality.  They are permutative under both 
$\oplus$ and $\otimes$, $0$ is a strict two-sided zero object for the
functor $\otimes$, and the right distributivity law holds strictly,
so that 
\[  (A\oplus B)\otimes C = (A\otimes C)\oplus (B\otimes C). \]  
This equality must be a permutative functor with respect to $\oplus$,
so that $c_{\oplus}\otimes \id = c_{\oplus}$.   
The left distributivity law $\de$ is specified by (\ref{dis2}), with
$\de'=\text{id}$, and cannot be expected to hold strictly.  Only one
additional coherence diagram is required to commute, namely
\[  \xymatrix{
(A\oplus B)\otimes (C\oplus D) \ar@{=}[dd]  \ar[r]^-{\de}
& ((A\oplus B)\otimes C)\oplus ((A\oplus B)\otimes D) \ar@{=}[d] \\
& (A\otimes C) \oplus (B\otimes C)\oplus (A\otimes D)\oplus (B\otimes D)
\ar[d]^{\Id\oplus c_{\oplus} \oplus \Id}\\
(A\otimes(C\oplus D))\oplus (B\otimes(C\oplus D)) \ar[r]_-{\de\oplus \de} 
& (A\otimes C)\oplus(A\otimes D)\oplus (B\otimes C)\oplus (B\otimes D).\\} \]
Since bipermutative categories are a specialization of Laplaza's symmetric 
bi\-mon\-oid\-al categories, his work resolves their coherence problem.  
The assymmetry in the distributive laws is intrinsic, and the strictness of
the right rather than the left distributivity law meshes with our 
use of lexicographic orderings in specifying the notion of an
action of an operad pair.  It is proven in \cite[VI\S3]{MQR} 
that any (small) strong symmetric bimonoidal category is equivalent 
to a bipermutative category, so that there is no loss of generality
in restricting attention to bipermutative categories.

\begin{sch}
Regrettably, the term bipermutative category was redefined in \cite{EM} to mean a weaker and definitely inequivalent notion, which we call 
a lax bipermutative category.  It has two permutative structures, 
but it only
has lax distributivity maps.  That is, it has a map like that of (\ref{dis}) 
and therefore its companion map of (\ref{dis2}), but with the arrows 
pointing in the opposite direction.  It is stated on \cite[p. 178]{EM} that ``Laplaza's symmetric bimonoidal categories are more general even than our 
bipermutative categories, and since they can be rectified to equivalent
bipermutative categories in May's sense, so can ours.'' This statement
is wrong on two counts.   Lax bipermutative categories  
are not special cases of Laplaza's semi op-lax symmetric bimonoidal categories,
and neither the latter nor the former 
can be rectified unless the distributivity maps are isomorphisms.  We note that no precise definition or coherence theorem
has been formulated for lax symmetric bimonoidal categories, and it is unclear that such objects can be rectified to the lax
bipermutative categories of \cite{EM}.
\end{sch}

From the point of view of our applications, these differences do not much matter.  
The interesting examples are strong symmetric bimonoidal 
and can be rectified to bipermutative categories as originally defined. 
The latter give rise to $(\sF\int\sF)$-spaces, as I 
recall in the next section.  In fact, as I will explain, any sensible notion
of lax or op-lax bipermutative category works for that. By the earlier sections 
of this paper, $(\sF\int\sF)$-spaces give rise to $E_{\infty}$ ring spaces.  By the
theory recalled in the prequel \cite{Prequel}, $E_{\infty}$ ring spaces give 
rise to $E_{\infty}$ ring spectra. 

From the point of view of mathematical philosophy and comparisons of
constructions, these differences do matter.  The theory of \cite{EM}
constructs symmetric ring spectra from lax bipermutative categories
and, as it stands, cannot recover 
the applications of \cite{MQR} (and other more recent 
applications), that depend on the use of $E_{\infty}$ ring spaces.  
We need a comparison theorem to the effect that if we start with a 
bipermutative category and process it to an $E_{\infty}$ ring spectrum 
and thus to a commutative $S$-algebra by going through the theory here 
and in \cite{Prequel}, then the 
result is equivalent to what we get by using \cite{EM} to construct a 
symmetric ring spectrum and converting that to a commutative $S$-algebra.
This should be true, but it is not at all obvious. 

Before continuing, we highlight the mistake in \cite{MQR} which led to
the need for the theory that we are describing here.

\begin{sch}  In \cite[VI.2.3, VI.2.6, and VI.4.4]{MQR}, it is claimed 
that $(\sM,\sM)$ and $(\sD,\sD)$ are operad pairs and that $(\sD,\sD)$ acts
on the classifying spaces of bipermutative categories.  These assertions are
incorrect, as is explained in detail in \cite[App A]{Mult}, and for this
reason there seems to be no elementary shortcut showing that the classifying
spaces of bipermutative categories are $E_{\infty}$ ring spaces.   The use of 
$\sD$ alone in the theory of permutative categories is unaffected by the mistake. 
\end{sch} 

\section{The construction of $(\sF\int\sF)$-categories from bipermutative 
categories}\label{Doperad}

There are notions of lax, strong, and strict morphisms between 
symmetric bimonoidal categories, in analogy with the corresponding
notions for symmetric monoidal categories.  Again, the strict notion
is only interesting in the bipermutative case.  We recall a problem 
that was left open in \cite[p. 16]{Mult}.

\begin{conj}\label{yuck} There is a functor on the bipermutative 
category level that replaces lax morphisms by strict morphisms, in a sense analogous
to the corresponding result \cite[4.3]{Pair} for permutative categories.
\end{conj}

In analogy with the permutative category situation, there are two functors,
one combinatorial and one conceptual, that construct $(\sF\int\sF)$-categories
from bipermutative categories.  The combinatorial construction is due to
Woolfson \cite{Woolf} and entails use of a more complicated category that 
contains $\sF\int\sF$.  It is spelled out in detail in \cite[App D]{Mult}.  
It is only functorial on strict morphisms.  In the absence of a proof of
Conjecture \ref{yuck}, this makes it less useful than its permutative category
analogue.

The conceptual construction is given in \cite[\S4]{Mult} and is again
an application of Street's first construction from \cite{Street}.  A
detailed restatement of the properties of the construction is given in
\cite[3.4]{Pair}.  In brief, for any (small) category $\sG$, it gives a 
functor from the category of either lax or op-lax functors $\sG\rtarr \sC\!at$ 
and lax or op-lax natural transformations to the category of genuine functors 
and genuine natural transformations $\sG\rtarr \sC\!at$, together with a comparison of the input and output up to ``natural homotopy''.  \footnote{It is generally understood in 
bicategory theory that lax functors $F$ should have comparison natural transformations 
$F(\ps)\com F(\ph) \rtarr F(\ps\com\ph)$; op-lax functors should have the arrows
reversed.  Street \cite{Street} uses lax functors and calls them that; in view of 
the freedom to replace $\sG$ by $\sG^{op}$, his construction applies equally well to 
op-lax functors.  Unfortunately, in \cite{Pair, Mult}, I used op-lax functors but 
called them lax functors.  I'll call them op-lax functors here.  The natural homotopies of
\cite{Pair, Mult} are special cases of what are called ``modifications'' in the bicategorical literature.}
  
To apply this general categorical construction to our situation, we need only construct 
an op-lax (or lax) functor $A\colon \sF\int\sF\rtarr \sC\!at$ from a bipermutative category $\sA$.  That is very easy to do. We recall the details from \cite[\S3]{Mult} to emphasize the role of the distributivity law and explain why a lax or op-lax law would work just as well as a strict or strong law.\footnote{I learned this from Michael Shulman.}  We start by specifying 
$A(\mathbf{n};S) = \sA^{s_1}\times \cdots\times \sA^{s_n}$ on objects, where 
$A(0;*)$ is the trival category $*$ and 
$\sA^0$ is the trivial category $0$. For a morphism $(\ph;\chi)\colon (\mathbf{m};R)\rtarr (\mathbf{n};S)$ in $\sF\int\sF$, we specify the functor
\[ A(\ph;\chi)\colon A(\mathbf{m};R)\rtarr A(\mathbf{n};S) \]
by the formula
\[A(\ph;\chi) (\times_{i=1}^m \times_{u=1}^{r_i} a_{i,u})
= \times_{j=1}^n \times_{v=1}^{s_j} \bigoplus_{\chi_j(U) = v} \bigotimes_{\ph(i)=j}a_{i,u_i} \]
on both objects and morphisms.  Here $U$ runs over the lexicographically ordered set (l.o.s)  of 
sequences with $i^{th}$ term $u_i$ satisfying $1\leq u_i\leq r_i$ for 
$i\in \ph^{-1}(j)$;
this set can be identified with $\sma_{\ph(i)=j}\mathbf{r_i}-\{0\}$. 
That is the same formula that we would have used if we had given a complete proof of
Proposition \ref{FFX}, describing rig spaces as $\sF\int\sF$-spaces explicitly.  In that context, we would have strict commutativity and distributivity and the formula would give a functor 
$\sF\int\sF\rtarr \sU$.  In the present context, we have coherence isomorphisms that give lax functoriality. Note
first that $A$ takes identity morphisms to identity functors.  However, for a second
morphism $(\ps;\om)\colon (\mathbf{n};S)\rtarr (\mathbf{p};T)$ in $\sF\int\sF$, we have
\[ A(\ps\com\ph;\xi)(\times_{i=1}^m\times_{u=1}^{r_i} a_{i,u})
= \times_{k=1}^p\times_{w=1}^{t_k} \bigoplus_{\xi_k(Y)=w}
\bigotimes_{(\ps\ph)(i)=k} a_{i,y_i},\]
where $\xi_k = \om_{k}\com (\sma_{\psi(j)=k}\ch_j)\com \si_k(\ps,\ph)$ and $Y$ runs
through the l.o.s of sequences with $1\leq y_i\leq r_i$ for $i\in (\ps\ph)^{-1}(k)$,
regarded as elements of $\sma_{\ps\ph(i) = k} \mathbf{r_i}$,
whereas
\[ A(\ps;\om)A(\ph;\ch)(\times_{i=1}^m\times_{u=1}^{r_i} a_{i,u})
= \times_{k=1}^p\times_{w=1}^{t_k}\bigoplus_{\om_k(V)=w} \bigotimes_{\ps(j) = k}
(\bigoplus_{\ch_j(U) = v_j} \bigotimes_{\ph(i) = j} a_{i,u_i}), \]
where $U$ run through the l.o.s. of sequences with $1\leq u_i\leq r_i$
for $i\in \ph^{-1}(j)$ and $V$ runs through the l.o.s. of sequences with 
$1\leq v_j\leq s_j$ for $j\in \ps^{-1}(k)$,
regarded as elements of
$\sma_{\ph(i)=j}\mathbf{r_i}$ and $\sma_{\psi(j) = k}\mathbf{s_j}$ respectively.  The 
commutativity isomorphisms $c_{\oplus}$ and $c_{\otimes}$, together with the 
strict right distributivity law, induce a natural isomorphism
\[\si((\ps;\om),(\ph;\chi))\colon A(\ps\com\ph;\xi) \rtarr A(\ps;\om)A(\ph;\ch).\]
The coherence in the definition of a bicategory gives the coherence with respect
to the associativity and unity of composition that are implicit in the assertion 
that this definition does give an op-lax functor.   

Clearly, we can reverse the arrow, and 
then we have a lax rather than op-lax functor.  With the proper specification of
coherence data, dictated by the requirement that the definition give a lax or
op-lax functor, we see that there is no need for a strict or even a strong 
distributivity law.  We conclude that, with proper definitions, we can obtain
a lax or op-lax functor from a lax or op-lax bipermutative category.  Street's
construction applies to rectify either to a (special) functor 
$\sF\int\sF\rtarr \sC\!at$.  Thus we first construct a lax or op-lax functor that 
uses actual cartesian products on objects, and we then 
use Street's construction to convert it to a genuine functor, but one that no longer 
uses actual cartesian products of objects.  Street's construction is ideally suited to convert the kind of structured categories that we encounter in nature to the kind of structured categories that we know how to convert to $E_{\infty}$ ring spaces after passage to their classifying spaces.

\section{Appendix A. Generalities on monads}\label{MonadAp}

To make this paper reasonably self-contained, we repeat some results
from \cite[p. 219]{MT} and \cite[\S5]{Mult}; the elementary categorical
proofs may be found there. 

Let $L\colon \sW\rtarr \sV$ and $R\colon \sV\rtarr \sW$
be an adjoint pair of functors with counit $LR = \text{Id}$ and unit
$\de\colon \text{Id}\rtarr RL$. 
We have a pair of propositions and corollaries relating monad 
structures on functors 
$C\colon \sV\rtarr \sV$ and $D\colon \sW\rtarr \sW$.  They differ
due to the assymmetry of our assumptions on $L$ and $R$.
The next result is \cite[5.1]{Mult}.

\begin{prop}  Let $(C,\mu,\et)$ be a monad on $\sV$,
let $(F,\rh)$ be a (right) $C$-functor in some category $\sV'$, and let
$(X,\xi)$ be a $C$-algebra in $\sV$. Define $D = RCL$.
\begin{enumerate}[(i)]
\item $D$ is a monad on $\sW$ with unit and product the composites
\[ \xymatrix@1{
\text{Id} \ar[r]^-{\de} & RL \ar[r]^-{R\et L} & RCL = D \\} \]
\[ \xymatrix@1{
DD= RCLRCL = RCCL \ar[r]^-{R\mu L} & RCL. \\} \]
\item $FL$ is a $D$-functor in $\sV'$ with right action
\[ \rh L\colon FLD = FLRCL = FCL \rtarr FL.  \]
\item $RX$ is a $D$-algebra in $\sW$ with action
\[ R\xi\colon  DRX = RCLRX = RCX \rtarr RX.  \]
\end{enumerate}
\end{prop}

In the present generality, we state results about bar constructions
simplicially.  After geometric realization in our space level situations, 
they give corresponding
results about the actual bar constructions of interest.

\begin{cor} The simplicial two-sided bar construction satisfies
\[ B_*(F,C,X) = B_*(FL,D,RX) \]
for a $C$-algebra $X$ and $C$-functor $F$, where $D = RCL$.
\end{cor}

The following result combines the two results \cite[5.2 and 5.3]{Mult}.

\begin{prop}\label{omniold} Let $(D,\nu, \ze)$ be a 
monad on $\sW$.  Define $C = LDR$ and let $\bar{\de}\colon D\rtarr RCL$ 
denote the common composite in the diagram
\[  \xymatrix{
D\ar[r]^-{D\de} \ar[d]_{\de D} \ar@{-->}[dr]^{\bar{\de}} 
& DRL \ar[d]^{\de DRL}\\
RLD\ar[r]_-{RLD\de} & RLDRL.\\}  \]
Assume that one of the following two natural maps is an isomorphism:
\[ \de DR =\bar{\de}R\colon DR\rtarr RC \ \ \ \text{or}\ \ \ 
LD\de = L\bar{\de}\colon LD\rtarr CL. \]
\begin{enumerate}[(i)]
\item $C$ is a monad on $\sV$ with unit and product the composites
\[ \xymatrix@1{
\text{Id} = LR\ar[r]^-{L\ze R} & LDR \\} \]
and
\[ \xymatrix@1{
CC = LDRLDR \ar[rr]^-{(LD\de DR)^{-1}}  && LDDR \ar[r]^-{L\nu R} & LDR = C,\\} \]
and $\bar{\de}\colon D\rtarr RCL$ is a map of monads on $\sW$.
\item If $(F,\rh)$ is a $C$-functor, then $(FL,\rh L\com FL\bar{\de})$
is a $D$-functor. In particular, $RCL$ is a $D$-functor and 
$\bar{\de}\colon D\rtarr RCL$ is a map of $D$-functors.
\item If $(X,\xi)$ is a $C$-algebra, then $(RX,R\xi\com \bar{\de}R)$ is a
$D$-algebra. In particular, for $Y\in\sW$, $RCLY$ is a $D$-algebra and
$\bar{\de}\colon DY\rtarr RCL Y$ is a map of $D$-algebras. 
\item If $(RX,\ps)$ is a $D$-algebra, then $(X, L\ps)$ is a $C$-algebra,
and $R$ embeds $C[\sV]$ into $D[\sW]$ as the full subcategory of $D$-algebras of the form $RX$. 
\item When $LD\de\colon LD\rtarr CL$ is an isomorphism, if $(Y,\ps)$ is a 
$D$-algebra in $\sW$, then $(LY,L\ps\com (LD\de)^{-1})$ is a $C$-algebra 
in $\sV$ and $\de\colon Y\rtarr RLY$ is a map of $D$-algebras.
\end{enumerate}
\end{prop}

\begin{cor}\label{coromni1} Let $D$ be a monad on $\sW$ and let $C = LDR$. 
\begin{enumerate}[(i)]
\item If $\bar{\de}R\colon DR\rtarr RC$ is an isomorphism, then
\[ B_*(GR,C,X)\iso B_*(G,D,RX) \]
for a $C$-algebra $X$ and $D$-functor $G$ and therefore
\[ B_*(F,C,X)\iso B_*(FL,D,RX) \]
for a $C$-functor $F$. 
\item If $LD\de\colon LD\rtarr CL$ is an isomorphism, then
\[ B_*(F,C,LY)\iso B_*(FL,D,Y) \]
for a $D$-algebra $Y$ and a $C$-functor $F$.
\end{enumerate}
\end{cor}

Recall from \cite[9.8]{Geo} that we always have a map 
\[ \epz_*\colon B_*(C,C,X)\rtarr X_* \]
of simplicial $C$-algebras that is a simplicial homotopy equivalence,
where $X_*$ is the constant simplicial object at $X$.  

\begin{cor}\label{coromni2} Under the hypotheses of Proposition \ref{omniold}, 
\[  \bar{\de}_*= B_*(\bar{\de},\id,\id)\colon B_*(D,D,Y)
\rtarr B_*(RCL,D,Y) = RB_*(CL,D,Y) \]
is a map of simplicial $D$-algbras.  If $\bar{\de}R\colon DR\rtarr RC$
is an isomorphism and $Y = RX$ for a $C$-algebra $X$, then the diagram
\[ \xymatrix@1{
Y_* & B_*(D,D,Y)\ar[r]^-{\bar{\de}_*} \ar[l]_-{\epz_*} & RB_*(CL,D,Y)\\} \]
of simplicial $D$-algebras is obtained by applying $R$ to the evident diagram
\[ \xymatrix@1{
X_* & B_*(C,C,X)\ar[r]^-{\iso} \ar[l]_-{\epz_*} & B_*(CL,D,RX)\\} \]
of simplicial $C$-algebras.
\end{cor}

\section{Appendix B. Monads and distributivity}\label{BeckAp}

Consider two monads, $(C,\mu_{\oplus},\et_{\oplus})$ and 
$(G,\mu_{\otimes},\et_{\otimes})$, on the same category $\sV$.  
As the notation indicates, we think of $C$ as ``additive'' and $G$ as ``multiplicative''.  We want to understand a monadic distributivity 
law for an action of 
$G$ on $C$.  This was obtained in an elegant paper of Beck \cite{Beck},
as I only learned after reproducing many of its results in the course
of working out multiplicative infinite loop space theory \cite[\S5]{Mult}.   Since this
theory is central to understanding, we repeat it here in abbreviated 
form, referring the reader to Beck \cite{Beck} for detailed verifications.

Let $C[\sV]$ and $G[\sV]$ denote the categories of $C$-algebras and
$G$-algebras in $\sV$. 

\begin{defn}\label{CGact} An action of $G$ on $C$ is a structure of monad 
on $G[\sV]$ induced by the monad $C$ on $\sV$. In detail, for an action 
 of $G$ on $X$, there is a prescribed functorial induced action of $G$ 
 on $CX$ (and thus on $CCX$ by iteration) such that
$\et_{\oplus}\colon X\rtarr CX$ and $\mu_{\oplus}\colon CCX\rtarr CX$
are maps of $G$-algebras. 
\end{defn}

Recall that the following diagram commutes 
for composable pairs of functors
$(B,A)$ and $(D,C)$ and for natural transformations $\al\colon A\rtarr C$
and $\be\colon B\rtarr D$.
\[ \xymatrix{
& BC \ar[dr]^{\be C} & \\
BA \ar[ur]^{B\al} \ar[dr]_{\be A}  & & DC \\
& DA \ar[ur]_{D\al} &\\} \]
In the categorical literature the common composite is generally
written $\al\be$ or $\be\al$. It is just the horizontal composition
of the $2$-category $\sC\!at$, but we shall be explicit.

\begin{thm}\label{Beckthm} The following data relating the monads $C$ and $G$ 
are equivalent.
\begin{enumerate}[(i)]
\item An action of $G$ on $C$.
\item A natural transformation $\mu\colon CGCG\rtarr CG$ with 
the following properties.
\begin{enumerate}[(a)]
\item $(CG,\mu,\et)$ is a monad on $\sV$, where 
$\et = \et_{\oplus}G \com \et_{\otimes}\colon \text{Id}\rtarr CG$.
\item $C\et_{\otimes}\colon C\rtarr CG$ and $\et_{\oplus}G\colon G\rtarr CG$
are maps of monads.
\item  The following composite is the identity natural transformation.
\[ \xymatrix@1{
CG \ar[r]^-{C\et_{\oplus} G} & CCG\ar[r]^-{C\et_{\otimes}CG } & CGCG
\ar[r]^-{\mu} & CG \\} \]
\end{enumerate}
\item A natural transformation $\rh\colon GC\rtarr CG$ such that
the following two diagrams commute.
\[\xymatrix{
& G \ar[dl]_{G\et_{\oplus}} \ar[dr]^{\et_{\oplus}G} & \\
 GC \ar[rr]^-{\rh} & & CG \\ 
& C \ar[ul]^{\et_{\otimes}C} \ar[ur]_{C\et_{\otimes}} & \\ } \]
and
\[\xymatrix{
GCC \ar[d]_{G\mu_{\oplus}} \ar[r]^-{\rh C} & CGC \ar[r]^{C\rh} 
& CCG \ar[d]^{\mu_{\oplus}G} \\ 
GC \ar[rr]^-{\rh} &   & CG \\ 
GGC \ar[r]_-{G\rh} \ar[u]^{\mu_{\otimes}C} & GCG \ar[r]_-{\rh G} & CGG.
\ar[u]_{C\mu_{\otimes}} \\ }\]
\end{enumerate}
When given such data, the category $C[G[\sV]]$ of $C$-algebras in
$G[\sV]$ is isomorphic to the category $CG[\sV]$ of $CG$-algebras
in $\sV$. 
\end{thm}
\begin{proof}[Sketch proof] Details are in \cite{Beck}. We relate
(i) to (iii) and (ii) to (iii).  Given the 
data of (i), we obtain the data of (iii) by defining $\rh$ to be the 
composite 
\[ \xymatrix@1{ 
GC\ar[r]^-{GC\et_{\otimes}} & GCG \ar[r]^-{\xi} & CG, \\} \]
where, for $X\in \sV$, $\xi$ is the action of $G$ on $CGX$ 
induced from the canonical action of $G$ on $GX$.  Given
the map $\rh$ as in (iii) and given a $G$-algebra $(X,\xi)$,
the following composite specifies a natural action of $G$ on $CX$
that satisfies (i). 
\[ \xymatrix@1{ 
GCX \ar[r]^-{\rh} & CGX \ar[r]^-{C\xi} & CX \\} \]
Given $\mu$ satisfying (ii), the following composite is a 
map $\rh$ satisfying (iii). 
\[ \xymatrix@1{
GC \ar[r]^-{GC\et_{\otimes}} & GCG\ar[r]^-{\et_{\oplus}GCG} &
CGCG \ar[r]^-{\mu} & CG \\} \]
Given $\rh$ satisfying (iii), the following 
composite is a map $\mu$ satisfying (ii).
\[ \xymatrix@1{ 
CGCG \ar[r]^-{C\rh G} & CCGG \ar[r]^-{CC\mu_{\otimes}} 
& CCG \ar[r]^-{\mu_{\oplus}G} & CG\\} \]
Given these equivalent data, a $C$-algebra $(X,\xi,\tha)$ in
$G[\sV]$ determines a $CG$-algebra $(X,\ps)$ in $\sV$ by letting
$\ps$ be the composite
\[\xymatrix@1{
CGX \ar[r]^-{C\xi} & CX \ar[r]^-{\tha} & X,\\ }\]
and a $CG$-algebra $(X,\ps)$ determines a $C$-algebra
$(X,\xi,\tha)$ in $G[\sV]$ by letting $\xi$ and $\tha$ be
the pullbacks of $\ps$ along the maps of monads $\et_{\oplus}G$
and $C\et_{\otimes}$ of (iib). 
\end{proof}

\end{document}